\documentclass{siamart1116}
\usepackage{amsmath}
\usepackage{amsfonts}
\usepackage{algorithm}
\usepackage{algorithmic}
\usepackage{url}
\usepackage{subfigure}

\def\noprint#1{}

\newtheorem{assumption}{Assumption}

\newcommand{\lambdamin}{\lambda_{\mbox{\rm\scriptsize{min}}}}

\newcommand{\R}{\mathbb{R}}

\newcommand{\eps}{\epsilon}
\newcommand{\epsg}{\epsilon_g}
\newcommand{\epsH}{\epsilon_H}
\newcommand{\flow}{f_{\mbox{\rm\scriptsize low}}}
\newcommand{\jin}{j_{inr}}
\newcommand{\cin}{c_{in}}
\newcommand{\jir}{j_{inr}}
\newcommand{\cir}{c_{ir}}

\newcommand{\refer}[1]{\textcolor{black}{#1}}

\def\tto{\;{\lower 1pt \hbox{$\rightarrow$}}\kern -10pt
           \hbox{\raise 2.8pt \hbox{$\rightarrow$}}\;}

\title{Complexity Analysis of Second-Order Line-Search Algorithms for Smooth Nonconvex Optimization%
\thanks{Version of \today. \funding{Work of the first author was
    supported by Subcontract 3F-30222 from Argonne National
    Laboratory. Work of the second author was supported by NSF Awards
    IIS-1447449, 1628384, 1634597, and 1740707; 
AFOSR Award FA9550-13-1-0138; and Subcontract
    3F-30222 from Argonne National Laboratory.  Part of this work was done
      while the second author was visiting the Simons Institute for
      the Theory of Computing, and partially supported by the
      DIMACS/Simons Collaboration on Bridging Continuous and Discrete
      Optimization through NSF Award CCF-1740425.}}}

\author{Cl\'ement W. Royer\thanks{Wisconsin Institute of Discovery,
    University of Wisconsin, 330 N. Orchard St., Madison, WI
    53715. (\email{croyer2@wisc.edu}).} \and Stephen
  J. Wright\thanks{Computer Sciences Department, University of
    Wisconsin, 1210 W. Dayton St., Madison, WI
    53706. (\email{swright@cs.wisc.edu}).}}

\begin{document}

\maketitle

\begin{abstract}
There has been much recent interest in finding unconstrained local
minima of smooth functions, due in part of the prevalence of such
problems in machine learning and robust statistics. A particular
focus is  algorithms with good complexity
guarantees. Second-order Newton-type methods that make use of
regularization and trust regions have been analyzed from such a
perspective. More recent proposals, based chiefly on first-order
methodology, have also been shown to enjoy optimal iteration
complexity rates, while providing additional guarantees on
computational cost.

In this paper, we present an algorithm with favorable complexity
properties that differs in two significant ways from other recently
proposed methods. First, it is based on line searches only: Each step
involves computation of a search direction, followed by a backtracking
line search along that direction. Second, its analysis is rather
straightforward, relying for the most part on the standard technique
for demonstrating sufficient decrease in the objective from
backtracking.
In the latter part of the paper, we consider inexact computation of
the search directions, using iterative methods in linear algebra: the
conjugate gradient and Lanczos methods. We derive modified convergence
and complexity results for these more practical methods.
\end{abstract}



\section{Introduction} \label{sec:intro}

We consider the unconstrained optimization problem
\begin{equation} \label{eq:f}
\min \, f(x),
\end{equation}
where $f:\R^n \to \R$ is a twice Lipschitz continuously differentiable
function that is generally nonconvex.
Some algorithms for this problem seek points that nearly satisfy the
second-order necessary conditions for optimality, which are that 
$\nabla f(x^*)=0$ and $\nabla^2 f(x^*) \succeq 0$.
These iterative schemes terminate at an iterate $x_k$ for which
\begin{equation} \label{eq:eps2opt}
\left\| \nabla f(x_k)\right\| \le \epsilon_g \quad \mathrm{and} \quad
\lambda_{\min}(\nabla^2 f(x_k)) \ge -\epsilon_H,
\end{equation}
where $\epsilon_g,\epsilon_H \in (0,1)$ are (typically small)
prescribed tolerances. Numerous algorithms have been proposed in
recent years for finding points that satisfy \eqref{eq:eps2opt}, each with
a complexity guarantee, which is an upper bound on an index $k$ that
satisfies \eqref{eq:eps2opt}, in terms of $\epsg$, $\epsH$, and other
quantities. We summarize below the main results.

Classical second-order convergent trust-region
schemes~\cite{ARConn_NIMGould_PhLToint_2000} can be shown to satisfy
\eqref{eq:eps2opt} after at most $\mathcal{O}\left(
\max\left\{\epsilon_g^{-2}\,\epsilon_H^{-1},
\epsilon_H^{-3}\right\}\right)$
iterations~\cite{CCartis_NIMGould_PhLToint_2012a}. Cubic
regularization methods in their basic
form~\cite{CCartis_NIMGould_PhLToint_2011a} have better complexity bounds than  trust-region
schemes, requiring at most
$\mathcal{O}\left(\max\left\{\epsilon_g^{-2},\epsilon_H^{-3}\right\}\right)$
iterations. The difference can be explained by the restriction 
enforced by the trust-region constraint on the norm of the steps. 
Recent work has shown that it is possible to improve the
bound for trust-region algorithms using specific definitions of the
trust-region radius~\cite{SGratton_CWRoyer_LNVicente_2017}. The best
known iteration bound for a second-order algorithm (that is, an
algorithm relying on the use of second-order derivatives and
Newton-type steps) is $\mathcal{O}\left(\max\left\{\epsilon_g^{-3/2},
\epsilon_H^{-3}\right\}\right)$. This bound was established originally
(under the form of a global convergence rate) 
in~\cite{YuNesterov_BTPolyak_2006}, by considering cubic
regularization of Newton's method. The same result is achieved by the
adaptive cubic regularization framework under suitable assumptions on
the computed step~\cite{CCartis_NIMGould_PhLToint_2012a}. Recent
 proposals have shown that the same bound can be attained
by algorithms other than cubic regularization. A modified trust-region
method~\cite{FECurtis_DPRobinson_MSamadi_2017a}, a variable-norm
trust-region scheme~\cite{JMMartinez_MRaydan_2017}, and a quadratic
regularization algorithm with cubic descent
condition~\cite{EGBirgin_JMMartinez_2017} all achieve the same bound.

When $\epsilon_g=\epsilon_H=\epsilon$ for some $\epsilon \in (0,1)$, all the 
bounds mentioned above reduce to $\mathcal{O}(\epsilon^{-3})$. It has been 
established that this order is sharp for the class of second-order
methods~\cite{CCartis_NIMGould_PhLToint_2012a}, and it can be proved
for a wide range of algorithms that make use of second-order derivative
information; see~\cite{GNGrapiglia_JYuan_YXYuan_2016b}. Setting
$\epsilon_H=\epsilon^{1/2}$ and $\epsilon_g=\epsilon$ for some
$\epsilon>0$ yields bounds varying between
$\mathcal{O}(\epsilon^{-3})$ and $\mathcal{O}(\epsilon^{-3/2})$, the
latter being again optimal within the class of second-order
algorithms~\cite{CCartis_NIMGould_PhLToint_2011c}.

A new trend in complexity analyses has emerged recently, that focuses
on measuring not just the number of iterations to achieve
\eqref{eq:eps2opt} but also the computational cost of the iterations.
Two independent proposals, respectively based on adapting accelerated
gradient to the nonconvex
setting~\cite{YCarmon_JCDuchi_OHinder_ASidford_2017} and approximately
solving the cubic
subproblem~\cite{NAgarwal_ZAllenZhu_BBullins_EHazan_TMa_2017}, require
$\mathcal{O}\left(\log\left(\frac{1}{\epsilon}\right)
\epsilon^{-7/4}\right)$ operations (with high probability, showing only 
dependency on $\epsilon$) to find a
point $x_k$ that satisfies
\begin{equation} \label{eq:semi2ndopt}
\left\| \nabla f(x_k)\right\| \le \epsilon \quad \mathrm{and} \quad
\lambda_{\min}(\nabla^2 f(x_k)) \ge -\sqrt{L_H\epsilon},
\end{equation}
with $L_H$ being a Lipschitz constant of the Hessian. 
The difference factor of $\epsilon^{-1/4}$ by comparison with the
complexities of the previous paragraph is due to the cost of computing
a negative eigenvalue of $\nabla^2 f(x_k)$ and/or the cost of solving
the linear system. A later proposal~\cite{YCarmon_JCDuchi_2017}
focuses on solving cubic subproblems via gradient descent, together
with an inexact eigenvalue computation: It
satisfies~\eqref{eq:semi2ndopt} in at most
$\mathcal{O}\left(\log\left(\frac{1}{\epsilon}\right)\epsilon^{-2}\right)$
with high probability. Another technique
\cite{CJin_RGe_PNetrapalli_SMKakade_MIJordan_2017} requires only
gradient computations, with noise being added to some iterates.  It
reaches with high probability a point satisfying~\eqref{eq:semi2ndopt}
in at most
$\mathcal{O}\left(\log^4\left(\frac{1}{\epsilon}\right)\epsilon^{-2}\right)$
iterations. Up to the logarithmic factor, this bound is characteristic
of gradient-type methods, but classical work establishes only
first-order guarantees~\cite{CCartis_NIMGould_PhLToint_2010}.
Although this setting is not explicitly addressed in the cited papers,
it appears that to reach an iterate satisfying~\eqref{eq:eps2opt} with
$\epsilon_g=\epsilon_H=\epsilon$, the methods studied
in~\cite{NAgarwal_ZAllenZhu_BBullins_EHazan_TMa_2017,
YCarmon_JCDuchi_OHinder_ASidford_2017} would require
$\mathcal{O}\left(\log\left(\frac{1}{\epsilon}\right)
\epsilon^{-7/2}\right)$ iterations, while the methods described
in~\cite{YCarmon_JCDuchi_2017}
and~\cite{CJin_RGe_PNetrapalli_SMKakade_MIJordan_2017} could require
$\mathcal{O}\left(\log\left(\frac{1}{\epsilon}\right)\epsilon^{-3}\right)$
and
$\mathcal{O}\left(\log^4\left(\frac{1}{\epsilon}\right)\epsilon^{-3}\right)$
iterations, respectively. Although these bounds look worse than those
of classical nonlinear optimization schemes, they are more
informative, in that they not only account for the number of outer
iterations of the algorithm, but also for the cost of performing 
each outer iteration (often measured in terms of the number of inner 
iterations, each of which has similar cost). We note, however, that unlike 
the classical complexity results, the newer procedures make use of
randomization, so the bounds typically hold only with high probability.

Our goal in this paper is to describe an algorithm that achieves
optimal complexity, whether measured by the number of iterations
required to satisfy the condition \eqref{eq:eps2opt} or by an estimate
of the number of fundamental operations required (gradient evaluations
or Hessian-vector multiplications). Each iteration of our algorithm
takes the form of a step calculation followed by a backtracking line
search. (To our knowledge, ours is the first line-search algorithm
that is endowed with a second-order complexity analysis.) The
``reference'' version of our algorithm is presented in
Section~\ref{sec:algoexact}, along with its complexity analysis. In
this version, we assume that two key operations --- solution of the linear
equations to obtain Newton-like steps and calculation of the most
negative eigenvalue of a Hessian --- are performed exactly.  In
Section~\ref{sec:algoinexact}, we refine our study by introducing
inexactness into these operations, and adjusting the complexity bounds
appropriately. Finally, we discuss the established results and their
practical connections in Section~\ref{sec:discuss}.

Throughout the paper, $\| \cdot \|$ denotes the Euclidean norm, unless 
otherwise indicated by a subscript. A vector $v$ will be called a \emph{unit 
vector} if $\|v\|=1$.

\section{A Line-Search Algorithm Based on Exact Step Computations}
\label{sec:algoexact}

We now describe an algorithm based on exact computation of search
directions, in particular, the Newton-like search directions and the
eigenvector that corresponds to the most negative eigenvalue of the
Hessian.

\subsection{Outline}

We use a standard line-search
framework~\cite[Chapter~3]{JNocedal_SJWright_2006}. Starting from an
initial iterate $x_0$, we apply an iterative scheme of the form
$x_{k+1}=x_k+\alpha_k d_k$, where $d_k$ is a chosen search direction
and $\alpha_k$ is a step length computed by a backtracking line-search
procedure.

Algorithm~\ref{algo:sorn} defines our method. Each iteration begins by
evaluating the gradient, together with the curvature of the function
along the gradient direction. This information determines whether the
negative gradient direction is a suitable choice for search direction
$d_k$, and if so, what scaling should be applied to it. If not, we 
compute the minimum eigenvalue of the Hessian. The corresponding eigenvector 
is used as the search direction whenever the
eigenvalue is sufficiently negative. Otherwise, we compute a
Newton-like search direction, adding a regularization term if needed
to ensure sufficient positive definiteness of the coefficient
matrix. There are a total of five possible choices for the search
direction $d_k$ (including two different scalings of the negative
gradient). Table~\ref{tab:stepsalgo} summarizes the various steps that
can be performed and the conditions under which those steps are
chosen.

\begin{table}[ht!]
	\begin{center}
	\begin{tabular}{|l|l|c||ll|}
		\hline
		\multicolumn{3}{|c||}{Context} &Direction &Decrease \\
		\hline
		$\|g_k\| = 0$ &- &$\lambda_k<-\epsH$ 	&$v_k$	
		&Lemma~\ref{lemma:eigsteplength} \\ \hline
		& $R_k < -\epsH$ & & $R_k g_k/\|g_k\|$ &Lemma~\ref{lemma:eigsteplength}\\ 
		\hline
		$\| g_k \| > \epsg$ & $R_k \in [-\epsH,\epsH]$ & & $-g_k/\|g_k\|^{1/2}$ 
		&Lemma~\ref{lemma:gradsteplength}\\ \hline
		$\| g_k \| \le \epsg$ & $R_k \in [-\epsH,\epsH]$ & $\lambda_k<-\epsH$ &$v_k$ 
		& Lemma~\ref{lemma:eigsteplength}\\
		$\| g_k \| \le \epsg$ & $R_k \in [-\epsH,\epsH]$ 
		&$\lambda_k \in [-\epsH,\epsH]$ & $d^r_k$ 
		&Lemma~\ref{lemma:regnewtsteplength}\\ \hline
		$\| g_k \| > \epsg$ & $R_k > \epsH$ & $\lambda_k<-\epsH$ & $v_k$ 
		& Lemma~\ref{lemma:eigsteplength}\\
		$\| g_k \| > \epsg$ & $R_k > \epsH$ & $\lambda_k \in [-\epsH,\epsH]$ 
		& $d^r_k$ & Lemma~\ref{lemma:regnewtsteplength}\\
		$\| g_k \| > \epsg$ & $R_k > \epsH$ & $\lambda_k >\epsH$ & $d^n_k$ 
		& Lemma~\ref{lemma:newtsteplength} \\ \hline
	\end{tabular}
	\end{center}
	\vspace*{0.5ex}
	\label{tab:stepsalgo}
	\caption{Steps and associated decrease lemmas for Algorithm~\ref{algo:sorn}.}
\end{table}

Once a search direction has been selected, a backtracking line search is 
applied with an initial choice of $1$. A sufficient condition related to the 
cube of the step norm must be satisfied; see~\eqref{eq:lsdecrease}. Such a 
condition has been instrumental in the complexity analysis of recently 
proposed Newton-type methods achieving the best known iteration complexity
rates~\cite{EGBirgin_JMMartinez_2017,FECurtis_DPRobinson_MSamadi_2017a}.

\begin{algorithm}[H]
\caption{Second-Order Line Search Method}
\label{algo:sorn}
\begin{algorithmic}
\STATE \emph{Init.} Choose $x^0 \in \R^n$, $\theta \in (0,1)$, $\eta > 0$,
$\epsilon_g \in (0,1)$, $\epsilon_H \in (0,1)$;
\FOR{$k=0,1,2,\dotsc$}
\STATE \textbf{Step 1. (First-Order)} Set
$g_k = \nabla f(x_k)$;
\IF{$\|g_k\|=0$}
\STATE Go to Step 2;
\ENDIF
\STATE Compute $R_k = \frac{g_k^\top \nabla^2 f(x_k) g_k}{\|g_k\|^2}$;
\IF{$R_k<-\epsH$}
\STATE{Set $d_k = \frac{R_k}{\|g_k\|}g_k$ and go to Step LS;}
\ELSIF{$R_k \in [-\epsH,\epsilon_H]$ and  $\|g_k\| > \epsilon_g$}
\STATE{Set $d_k = -\frac{g_k}{\|g_k\|^{1/2}}$ and go to Step LS;}
\ELSE
\STATE{Go to Step 2;}
\ENDIF
\STATE \textbf{Step 2. (Second-Order)} Compute an eigenpair 
$(v_k,\lambda_k) \in \R^n \times \R$ where 
$\lambda_k=\lambda_{\min}(\nabla^2 f(x_k))$  and $v_k$ is such that 
\begin{equation} \label{eq:eigenvec}
	\nabla^2 f(x_k) v_k = \lambda_k v_k,\quad v_k^\top g_k \le 0,
	\quad \|v_k\|=[-\lambda_k]_+;
\end{equation}
\IF{$\| g_k \| \le \epsg$ and $\lambda_k \ge -\epsH$}
\STATE{\textbf{Terminate (or go to Local Phase)};}
\ELSIF{$\lambda_k < -\epsilon_H$}
\STATE{\textbf{(Negative Curvature)} Set $d_k=v_k$;}
\ELSIF{$\lambda_k > \epsilon_H$} 
\STATE \textbf{(Newton)} Set $d_k=d^n_k$, where 
\begin{equation} \label{eq:fullnewton}
	\nabla^2 f(x_k) d^n_k = -g_k;
\end{equation}
\ELSE
\STATE  \textbf{(Regularized Newton)} Set $d_k=d^r_k$, where
\begin{equation} \label{eq:regnewton}
	\left(\nabla^2 f(x_k) + 2 \epsilon_H I \right)d^r_k = -g_k;
\end{equation}
\ENDIF
\STATE Go to Step LS;
\STATE \textbf{Step LS. (Line Search)} Compute a step length 
$\alpha_k=\theta^{j_k}$, where $j_k$ is the smallest nonnegative integer such 
that
\begin{equation} \label{eq:lsdecrease}
	f(x_k + \alpha_k d_k) < f(x_k) - \frac{\eta}{6}\alpha_k^3 \|d_k\|^{3}
\end{equation}
holds, and set $x_{k+1} = x_k+\alpha_k d_k$.
\IF{$d_k=d^n_k$ or $d_k=d^r_k$ and $\| \nabla f(x_{k+1}) \| \le \epsg$}
\STATE{\textbf{Terminate (or  go to Local Phase)}};
\ENDIF
\ENDFOR
\end{algorithmic}
\end{algorithm}

At most one eigenvector computation and one linear system solve are
needed per iteration of Algorithm~\ref{algo:sorn}, along with a
gradient evaluation and the Hessian-vector multiplication required to
calculate $R_k$.

The algorithm contains two tests for termination, with the option of
switching to a ``Local Phase'' instead of terminating at a point that
satisfies approximate second-order conditions.  The Local Phase aims
for rapid local convergence to a point satisfying second-order
necessary conditions for a local solution; it is detailed in
Algorithm~\ref{algo:localphase}. Termination (or switch to the Local
Phase) occurs at an iteration $k$ at which an
\emph{$(\epsilon_g,\epsilon_H)$-approximate second-order critical
  point} is reached, according to the following definition:
	\begin{equation} \label{eq:optwcc}
		\min\left\{ \|g_k\|,\|g_{k+1}\|\right\} \le \epsilon_g, 
		\quad \mathrm{and} \quad \lambdamin (\nabla^2 f(x_k))  \ge -\epsilon_H,
	\end{equation}
where $g_k = \nabla f(x_k)$, etc.  
As we see below, the quantity
$\min\left\{ \|g_k\|,\|g_{k+1}\|\right\}$ arises naturally in the
decrease formula we establish for the steps computed by
Algorithm~\ref{algo:sorn}. In fact, for the methods we reviewed in
introduction, one observes that the decrease formulas obtained for
their steps either involve only
$\|g_k\|$~\cite{NAgarwal_ZAllenZhu_BBullins_EHazan_TMa_2017,
  YCarmon_JCDuchi_2017, YCarmon_JCDuchi_OHinder_ASidford_2017,
  CJin_RGe_PNetrapalli_SMKakade_MIJordan_2017,YuNesterov_BTPolyak_2006},
only $\|g_{k+1}\|$~\cite{EGBirgin_JMMartinez_2017,
  FECurtis_DPRobinson_MSamadi_2017a,JMMartinez_MRaydan_2017}, or the
minimum of the two
quantities~\cite{CCartis_NIMGould_PhLToint_2011b}.  The later case
appears due to the presence of both gradient-type (see
Lemma~\ref{lemma:gradsteplength}) and Newton-type steps (see
Lemmas~\ref{lemma:newtsteplength} and~\ref{lemma:regnewtsteplength}).

\begin{algorithm}[H]
\caption{Local Phase}
\label{algo:localphase}
\begin{algorithmic}
\LOOP
\STATE{Set $g_k=\nabla f(x_k)$;}
\IF{$\| g_k \| > \epsg$}
\STATE{Return to Algorithm~\ref{algo:sorn};}
\ENDIF
\STATE{Compute $\lambda_k$ and $v_k$ as in \eqref{eq:eigenvec};}
\IF{$\lambda_k<-\epsH$}
\STATE{Return to Algorithm~\ref{algo:sorn};}
\ELSIF{$\lambda_k \in (-\epsH,0]$}
\STATE{Set $d_k=d^r_k$ from \eqref{eq:regnewton};}
\ELSE
\STATE{Set $d_k=d^n_k$ from \eqref{eq:fullnewton};}
\ENDIF
\STATE{Perform backtracking line search as in Step LS of 
Algorithm~\ref{algo:sorn} to obtain $x_{k+1}$;}
\STATE{$k \leftarrow k+1$;}
\ENDLOOP
\end{algorithmic}
\end{algorithm}

The main convergence results of this section are complexity results on
the number of iterations or function evaluations required to satisfy
condition~\eqref{eq:optwcc} {\em for the first time}. 
(Algorithm~\ref{algo:localphase} makes provision for
re-entering the main algorithm, if the approximate second-order
conditions are violated at any point. This re-entry feature is not covered
by our complexity analysis.)

\subsection{Iteration Complexity}
\label{subsec:wccexactits}

We now establish a complexity bound for Algorithm~\ref{algo:sorn}, in
the form of the maximum number of iterations that may occur before the
Termination conditions are satisfied for the first time.  To this end,
we provide guarantees on the decrease that can be obtained for each of
the possible choices of search direction. 

In the rest of this paper, we make the following assumptions.

\begin{assumption} \label{assum:compactlevelset}
	The level set $\mathcal{L}_f(x_0) = \{x | f(x) \le f(x_0)\}$ is 
	a compact set.
\end{assumption}

\begin{assumption} \label{assum:fC22}
	The function $f$ is twice Lipschitz continuously differentiable on an
	open neighborhood of $\mathcal{L}_f(x_0)$, and we denote by $L_g$ and $L_H$ 
	the respective Lipschitz constants for $\nabla f$ and $\nabla^2 f$ on this 
	set.
\end{assumption}

By the continuity of $f$ and its derivatives,
Assumption~\ref{assum:compactlevelset} implies that there exist
$\flow \in \R$, $U_g > 0$ and $U_H > 0$ such that for every 
\refer{$x \in \mathcal{L}_f(x_0)$}, one has
\begin{equation}\label{eq:boundscompact}
	f(x) \geq \flow,\quad \|\nabla f(x)\| \leq U_g,
	\quad \|\nabla^2 f(x)\| \leq U_H.
\end{equation}
We point out that the choice $U_H=L_g$ is a valid one for theoretical
purposes. However, $U_H$ will serve as an explicit parameter of our
inexact method in Section~\ref{sec:algoinexact}, so we 
use separate notation, to allow $U_H$ to be an overestimate of $L_g$.

An immediate consequence of these assumptions is that for any $x$ and
$d$ such that Assumption~\ref{assum:fC22} is satisfied at $x$ and
$x+d$, we have
\begin{equation} \label{eq:LH}
f(x+d) \le f(x) + \nabla f(x)^Td + \frac12 d^T \nabla^2 f(x) d + 
\frac{L_H}{6} \| d\|^3.
\end{equation}

The following four technical lemmas derive bounds on the decrease
obtained from each type of step. The proofs are rather similar 
to each other, and follow the usual template for backtracking line-search 
methods.

We begin with negative curvature directions, showing that our choices
for initial scaling yield a decrease proportional to the cube of the
(negative) curvature in that direction.

\begin{lemma} \label{lemma:eigsteplength}
Under Assumption~\ref{assum:fC22}, suppose that the search direction
for the $k$-th iteration of Algorithm~\ref{algo:sorn} is chosen either
as $d_k = \frac{R_k}{\|g_k\|}g_k$ with $R_k < - \epsilon_H$ in Step 1
or $d_k = v_k$ in Step 2. Then the backtracking line search terminates
with step length $\alpha_k = \theta^{j_k}$ with $j_k \le
  j_e+1$, where
\begin{equation} \label{eq:eiglsits}	
		j_e := \left[ \log_{\theta}\left( \frac{3}{L_H+\eta}
                  \right) \right]_+,
\end{equation}
and the decrease in the function value resulting from the chosen
step length satisfies
\begin{equation} \label{eq:eigsteplength}
		f(x_k) - f(x_k+\alpha_k\,d_k) \; \geq \; c_e 
		\left[\frac{|d_k^\top \nabla^2 f(x_k) d_k|}{\|d_k\|^2}\right]^3,
\end{equation}	
with
\begin{equation*}
		c_e := \frac{\eta}{6}
                \min\left\{1,\frac{27\theta^3}{(L_H+\eta)^3}\right\}.
\end{equation*}
\end{lemma}
\begin{proof}
For the direction $d_k=R_k g_k / \| g_k\|$, we have
\[
d_k^T \nabla^2 f(x_k) d_k = R_k^2 \frac{g_k^T \nabla^2 f(x_k) g_k}{\|g_k\|^2} 
= R_k^3 = -\| d_k \|^3.
\]
For the other choice $d_k = v_k$,  we have
$d_k^T \nabla^2 f(x_k) d_k = \lambda_k^3 = -\|d_k \|^3$, so that in both cases 
we have
\begin{equation} \label{eq:dk3}
d_k^T \nabla^2 f(x_k) d_k = -\|d_k\|^3 \quad \mathrm{and} \quad
\frac{| d_k^T \nabla^2 f(x_k) d_k|}{\|d_k\|^2} = \|d_k\|.
\end{equation}
Thus, if the unit value $\alpha_k=1$ is accepted by 
\eqref{eq:lsdecrease}, the result \eqref{eq:eigsteplength} holds trivially.
				
Suppose now that the unit step length is not accepted. Then the choice
$\alpha=\theta^j$ does {\em not} satisfy the decrease
condition~\eqref{eq:lsdecrease} for some $j \ge 0$. Using
\eqref{eq:LH} and the definition of $d_k$, we obtain
\begin{align*}
		-\frac{\eta}{6}\alpha^3\|d_k\|^3 \leq f(x_k+\alpha d_k) - f(x_k) 
		&\leq \alpha g_k^\top d_k + \frac{\alpha^2}{2}d_k^\top \nabla^2 f(x_k) d_k
		+ \frac{L_H}{6}\alpha^3\|d_k\|^3 \\
		&\leq \frac{\alpha^2}{2}d_k^\top \nabla^2 f(x_k) d_k
		+ \frac{L_H}{6}\alpha^3\|d_k\|^3 \\
		&= -\frac{\alpha^2}{2}\|d_k\|^3 + \frac{L_H}{6}\alpha^3 \|d_k\|^3,
\end{align*}
where the last line follows from \eqref{eq:dk3}. Therefore, we have
\begin{equation} \label{eq:eigdecfalse}
		\alpha = \theta^j \geq \frac{3}{L_H+\eta},
\end{equation}
which holds only if $j \le j_e$ by definition of $j_e$. Thus, the line
search must terminate with \eqref{eq:lsdecrease} being satisfied for some 
value $j_k \le j_e+1$. Because the line search did not stop with step length
$\theta^{j_k-1}$, we must have
\begin{equation*}
		\theta^{j_k-1} \geq \frac{3}{L_H+\eta} \Rightarrow
		\theta^{j_k} \geq \frac{3\theta}{L_H+\eta}.
\end{equation*}
As a result, the decrease satisfied by the step $\alpha_k
d_k=\theta^{j_k}d_k$ is such that
\begin{equation*}
		f(x_k)-f(x_k+\alpha_k d_k) \; \geq \; \frac{\eta}{6}\theta^{3 j_k}\|d_k\|^3
		\; \ge \; \frac{\eta}{6}\frac{27\theta^3}{(L_H+\eta)^3}
		\left[\frac{|d_k^\top \nabla^2 f(x_k) d_k|}{\|d_k\|^2}\right]^3.
\end{equation*}
This inequality, together with the analysis for the case of
$\alpha_k=1$, establishes the desired result.
\end{proof}

The second result concerns use of the step $d_k= - g_k/\|g_k\|^{1/2}$
in the case in which the curvature of the function along the gradient
direction is small.
\begin{lemma} \label{lemma:gradsteplength}
	Let Assumptions~\ref{assum:compactlevelset} and~\ref{assum:fC22} hold. 
	Then, if at the $k$-th iteration of Algorithm~\ref{algo:sorn}, the 
	search direction is $d_k= -g_k / \|g_k\|^{1/2}$, the backtracking 
	line search terminates with step length $\alpha_k = \theta^{j_k}$, with
       $j_k \le j_g+1$, where
	\begin{equation} \label{eq:gradlsits}
		j_g := \left[\log_{\theta}\left( \min\left\{ \frac{5}{3},
		\sqrt{\frac{1}{L_H+\eta}}\right\}\min\left\{\epsilon_g^{1/2}
		\epsilon_H^{-1},1\right\}\right) \right]_+,
	\end{equation}	
	and the resulting step length $\alpha_k$ is such that 
	\begin{equation} \label{eq:gradsteplength}
		f(x_k) - f(x_k+\alpha_k d_k) \; \geq \; c_g \min\left\{
		\epsilon_g^{3}\epsilon_H^{-3},\epsilon_g^{3/2}\right\}
	\end{equation}
	where 
	\begin{equation*}
		c_g:=\frac{\eta}{6}\min\left\{1,\frac{\theta^3}{(L_H+\eta)^{3/2}},
		\frac{125\theta^3}{27}\right\}.
	\end{equation*}
\end{lemma}
\begin{proof}
	Recall that the choice $d_k=-g_k/\|g_k\|^{1/2}$ is adopted
        only when $\|g_k\| > \epsg$ and $|R_k| \le \epsH$. 
        If the unit step length
        $\alpha_k=1$ is accepted, we have
	\begin{equation*}	
		f(x_k) - f(x_k+d_k) \; \ge \; \frac{\eta}{6}\|d_k\|^3 \; = \;
		\frac{\eta}{6}\|g_k\|^{3/2} \; \ge 
		\frac{\eta}{6}\epsilon_g^{3/2},
	\end{equation*}
satisfying \eqref{eq:gradsteplength}.  Otherwise, it means that there
exists $j \geq 0$ for which the decrease
condition~\eqref{eq:lsdecrease} is not satisfied using the step size
$\theta^j$.  For such $j$, we have from \eqref{eq:LH} that 
	\begin{align*}
		-\frac{\eta}{6}\theta^{3j}\|g_k\|^{3/2} &\leq 
		f(x_k-\theta^j\|g_k\|^{-1/2}g_k) - f(x_k) \\
		&\leq -\theta^j\|g_k\|^{3/2} + 
		\frac{\theta^{2j}}{2}R_k\,\|g_k\|	+ \frac{L_H}{6}\theta^{3j}
		\|g_k\|^{3/2} \\
		&\leq -\theta^j\|g_k\|^{3/2} + 
		\frac{\theta^{2j}}{2}\epsilon_H\,\|g_k\|	+ \frac{L_H}{6}\theta^{3j}
		\|g_k\|^{3/2},
	\end{align*}
	which leads to
	\begin{equation} \label{eq:mildcurvcontrad}
		0 \leq \left[ -\frac{5}{6}\theta^j\|g_k\|^{3/2} + 
		\frac{\theta^{2j}}{2}\epsilon_H\,\|g_k\|\right] + \left[ 
		-\frac{1}{6}\theta^j\|g_k\|^{3/2} + \frac{L_H+\eta}{6}\theta^{3j}
		\|g_k\|^{3/2}\right].
	\end{equation}
	Therefore, at least one of the two terms between brackets must be nonnegative.
	If
	\[
		-\frac{5}{6}\theta^j\|g_k\|^{3/2} + 
		\frac{\theta^{2j}}{2}\epsilon_H\,\|g_k\| \geq 0,
	\]	
	we have $\theta^j \geq \frac{5}{3}\|g_k\|^{1/2}\epsilon_H^{-1}$. On the other 
	hand, if
	\[
		-\frac{1}{6}\theta^j\|g_k\|^{3/2} + \frac{L_H+\eta}{6}\theta^{3j}
		\|g_k\|^{3/2} \geq 0,
	\]
	then $\theta^j \geq \sqrt{\frac{1}{L_H+\eta}}$. Putting the two bounds 
	together, we have that
	\begin{subequations} \label{eq:boundsteplenghtmildcurvcontrad}
	\begin{align}
	\label{eq:hy.a}
		\theta^j &\geq \min\left\{ \frac{5}{3}\|g_k\|^{1/2}\epsilon_H^{-1},
		\sqrt{\frac{1}{L_H+\eta}}\right\} \\
	\label{eq:hy.b}
		&\geq \min\left\{ 
		\frac{5}{3},\sqrt{\frac{1}{L_H+\eta}}\right\}
		\min\left\{\|g_k\|^{1/2}\epsilon_H^{-1},1\right\}  \\
	\label{eq:hy.c}
		&\geq \min\left\{ 
		\frac{5}{3},\sqrt{\frac{1}{L_H+\eta}}\right\}
		\min\left\{\epsilon_g^{1/2}\epsilon_H^{-1},1\right\}.
	\end{align}
	\end{subequations}
	Since $j > j_g$ contradicts~\eqref{eq:hy.c}, the line search terminates with 
	\eqref{eq:lsdecrease} being satisfied for some value $j_k \le j_g+1$.
	Since \eqref{eq:lsdecrease} did not hold for $\alpha = \theta^{j_k-1}$, 
	we have from \eqref{eq:hy.b} that
	\begin{equation*}
		\theta^{j_k} \; \geq \; \theta
		\min\left\{ \frac{5}{3},
		\sqrt{\frac{1}{L_H+\eta}}\right\}\min\left\{\|g_k\|^{1/2}
		\epsilon_H^{-1},1\right\}.
	\end{equation*}
	The decrease obtained by the step length $\alpha_k=\theta^{j_k}$ 
	thus satisfies
	\begin{eqnarray} \label{eq:graddecrease2}
		f(x_k)-f(x_k+\alpha_k d_k) &\ge &\frac{\eta}{6}\theta^{3j_k}
		\|d_k\|^3 \nonumber \\ 
		&\ge &\frac{\eta}{6}\left[\theta
		\min\left\{ \frac{5}{3},
		\sqrt{\frac{1}{L_H+\eta}}\right\}\right]^{3}\min\left\{
		\|g_k\|^{3/2}\epsilon_H^{-3},1\right\}\|g_k\|^{3/2} 
		\nonumber \\
		&\ge &\frac{\eta}{6}\left[\theta
		\min\left\{ \frac{5}{3},
		\sqrt{\frac{1}{L_H+\eta}}\right\}\right]^{3}\min\left\{
		\epsilon_g^{3}\epsilon_H^{-3},\epsilon_g^{3/2} \right\}.
	\end{eqnarray}	
	Thus \eqref{eq:gradsteplength} is also satisfied in the case of $\alpha_k<1$, 
	completing the proof.
\end{proof}

\medskip

Lemma~\ref{lemma:gradsteplength} describes the reduction that can be
achieved along the negative gradient direction when the curvature of
the function in this direction is modest. When this curvature is
significantly positive (or when this curvature is slightly positive
but the gradient is small), we compute the minimum Hessian eigenvalue
(Step 2) and consider other options for the search direction.

Our next result concerns the decrease that can be guaranteed by the
Newton step, when it is computed.
\begin{lemma} \label{lemma:newtsteplength}
Let Assumptions~\ref{assum:compactlevelset} and~\ref{assum:fC22} hold.
Suppose that the Newton direction $d_k=d^n_k$ is used at the $k$-th
iteration of Algorithm~\ref{algo:sorn}. Then the backtracking line
search terminates with step length $\alpha_k =\theta^{j_k}$, with 
$j_k \le j_n+1$, where
\begin{equation} \label{eq:newtlsits}	
j_n := \left[ \log_{\theta}\left( \sqrt{\frac{3}{L_H+\eta}}
  \frac{\epsilon_H}{\sqrt{U_g}}\right) \right]_+,
\end{equation}
and we have
\begin{equation} \label{eq:newtsteplength}
		f(x_k)-f(x_k+\alpha_k d_k) \; \geq \; c_{n}
		\min\left\{\|\nabla f(x_k+\alpha_k d_k)\|^{3/2},
		\epsilon_H^3\right\}.
\end{equation}
where
\begin{equation*}
c_{n} := \frac{\eta}{6}\min\left\{ \left[\frac{2}{L_H}\right]^{3/2},
\left[ \frac{3\theta}{L_H+\eta}\right]^3\right\}.
\end{equation*}
\end{lemma}
\begin{proof}
Note first that the Newton direction $d_k=d^n_k$ is computed only when
$\nabla^2 f(x_k)) \succ \epsH I$, so we have
\begin{equation} \label{eq:nn33}
\| d_k \| \le \| \nabla^2 f(x_k)^{-1} \| \| g_k \| \le U_g/\eps_H.
\end{equation}

Suppose first that the step length $\alpha_k=1$ satisfies the decrease
condition~\eqref{eq:lsdecrease}. Then from \eqref{eq:fullnewton} and
\eqref{eq:LH}, we have
\begin{align*}
		\left\| \nabla f(x_k+\alpha_k d_k) \right\| &=
		\left\| \nabla f(x_k+d_k) - \nabla f(x_k) + \nabla f(x_k) \right\| \\
		&=\left\| \nabla f(x_k+d_k)-\nabla f(x_k)-\nabla^2 f(x_k)d_k \right\| \leq 
		\frac{L_H}{2}\|d_k\|^2. 
\end{align*}
We thus have the following bound on the decrease obtained with the
unitary Newton step:
\begin{equation} \label{eq:decnewtunitstep}	
		f(x_k) - f(x_k+d_k) \; \geq \; \frac{\eta}{6}
		\left[\frac{2}{L_H}\right]^{3/2}
		\|\nabla f(x_k+d_k)\|^{3/2}.
\end{equation}
	
Suppose now that the unit step length does not allow for a sufficient
decrease as measured by~\eqref{eq:lsdecrease}. Then this condition must
fail for $\alpha_k = \theta^j$ for some $j \ge 0$. For this value, we have 
from \eqref{eq:LH}  that
\begin{align}
	\nonumber
	-\frac{\eta}{6}\theta^{3j}\|d_k\|^3 &\leq 
	f(x_k+\theta^j\,d_k) - f(x_k) \\
	\nonumber
	&\leq \theta^j g_k^T d_k + \frac{\theta^{2j}}{2}d_k^T \nabla^2 f(x_k)
	d_k + \frac{L_H}{6}\theta^{3j}\|d_k\|^3 \\
	\nonumber
	&\leq \theta^j\left(\frac{\theta^j}{2}-1\right) 
	d_k^T \nabla^2 f(x_k) d_k + \frac{L_H}{6}\theta^{3j}\|d_k\|^3 \\
	\nonumber
	&\leq -\frac{\theta^j}{2} d_k^T \nabla^2 f(x_k) d_k + 
	\frac{L_H}{6}\theta^{3j}\|d_k\|^3 \\
	\label{eq:nn34}
	&\leq -\frac{\theta^j}{2}\epsilon_H\|d_k\|^2 + 
	\frac{L_H}{6}\theta^{3j}\|d_k\|^3,
\end{align}
where we used $\nabla^2 f(x_k) \succeq \epsilon_H I$ for the final
inequality. This relation holds in particular for $j=0$, in which case
it gives
\begin{equation*} 
	-\frac{\eta}{6}\|d_k\|^3 \le -\frac{\epsilon_H}{2}\|d_k\|^2 + 
	\frac{L_H}{6}\|d_k\|^3
\end{equation*}		
leading to the following lower bound on the norm of the Newton step:
\begin{equation} \label{eq:unitnewtnonsucc}
		\|d_k\| \; \geq \; \frac{3}{L_H+\eta}\epsilon_H.
\end{equation}
More generally, for any integer $j$ such that the decrease condition
is not satisfied, we have from \eqref{eq:nn34} that
\begin{equation} \label{eq:tjge}
		\theta^j \; \geq \;  \sqrt{\frac{3}{L_H+\eta}} 
		\epsilon_H^{1/2}\|d_k\|^{-1/2}.
\end{equation}
For any $j > j_n$, the last inequality is violated since
\begin{equation*}
	\theta^j < \theta^{j_n} \le \sqrt{\frac{3}{L_H+\eta}} 
	\frac{\epsilon_H}{\sqrt{U_g}}
	= \sqrt{\frac{3}{L_H+\eta}} \epsilon_H^{1/2}
	\frac{\epsilon_H^{1/2}}{\sqrt{U_g}} \le \sqrt{\frac{3}{L_H+\eta}} 
	\epsilon_H^{1/2}\|d_k\|^{-1/2},
\end{equation*}
where we used \eqref{eq:nn33} for the final inequality.  This proves
that the condition~\eqref{eq:lsdecrease} will be satisfied by some 
$j_k \le j_n+1$. Since $\alpha = \theta^{j_k-1}$ does not fulfill the decrease 
requirement, it follows from \eqref{eq:tjge} that
\begin{equation*}
		\theta^{j_k} \; \geq \; \theta\sqrt{\frac{3}{L_H+\eta}} 
		\epsilon_H^{1/2}\|d_k\|^{-1/2}.
\end{equation*}
By substituting this lower bound into the sufficient decrease
condition, and then using \eqref{eq:unitnewtnonsucc}, we obtain
\begin{align*}
	f(x_k) - f(x_k+\alpha_k d_k) & = f(x_k) - f(x_k+\theta^{j_k} d_k)  \\
	&\geq \frac{\eta}{6}\theta^{3\,j_k}\|d_k\|^3 \\
	& \ge \frac{\eta}{6} \theta^3 \left[ \frac{3}{L_H+\eta} \right]^{3/2} 
	\epsH^{3/2} \| d_k \|^{-3/2} \| d_k \|^3 \\
	&\geq \frac{\eta}{6}\theta^3\left[\frac{3}{L_H+\eta}\right]^{3}\epsilon_H^3,
\end{align*}
where the final inequality is from \eqref{eq:unitnewtnonsucc}.  We
obtain the required result by combining this inequality with the bound
\eqref{eq:decnewtunitstep} for the case of $\alpha_k=1$.
\end{proof}

Our last intermediate result addresses the case of a regularized
Newton step.

\begin{lemma} \label{lemma:regnewtsteplength}
Let Assumptions~\ref{assum:compactlevelset} and~\ref{assum:fC22} hold.
Suppose that $d_k=d^r_k$ at the $k$-th iteration of
Algorithm~\ref{algo:sorn}. Then the backtracking line search
terminates with step length $\alpha_k = \theta^{j_k}$, with
$j_k \le j_r+1$, where
\begin{equation} \label{eq:regnewtlsits}
		j_r :=  \left[ \log_{\theta}\left( \frac{6}{L_H+\eta} 
		\frac{\epsilon_H^2}{U_g}\right) \right]_+,
\end{equation}	
and we have
\refer{
\begin{equation} \label{eq:regnewtsteplength}
		f(x_k)-f(x_k+\alpha_k d_k) \; \geq \; c_r\min\left\{
		\left\|\nabla f(x_k+\alpha_k d_k)\right\|^3\,\epsilon_H^{-3},
		\epsilon_H^3\right\},
\end{equation}
where 
\begin{equation*}
c_r := \frac{\eta}{6}\min\left\{ \left[ 
\frac{1}{1+\sqrt{1+L_H/2}}\right]^3, \left[
  \frac{6\theta}{L_H+\eta}\right]^{3} \right\}.
\end{equation*}
}
\end{lemma}
\begin{proof}
Note first that the regularized Newton step is taken only when
$\nabla^2 f(x_k) \succeq -\epsH I$. Thus the minimum eigenvalue of the
coefficient matrix in \eqref{eq:regnewton} is $\lambda_k + 2 \epsH \ge
\epsH$, and we have
	\begin{equation} \label{eq:upperboundnormdrk}
		\|d_k\| \; \leq \; \frac{\|g_k\|}{\lambda_k+2\epsilon_H} \; \leq \;
		\frac{\|g_k\|}{\epsilon_H} \; \leq \; \frac{U_g}{\epsilon_H}.
	\end{equation}

	Suppose first that the unit step is accepted. Then the
        gradient norm at the new point satisfies
	\begin{align*}
		\left\|\nabla f(x_k+d_k)\right\| &= 
		\left\|\nabla f(x_k+d_k)-\nabla f(x_k) + \nabla f(x_k) \right\| \\
		&= \left\|\nabla f(x_k+d_k)-\nabla f(x_k) - \nabla^2 f(x_k)d_k 
		- 2\epsilon_H d_k \right\| \\
		&\leq \frac{L_H}{2}\|d_k\|^2 + 2\epsilon_H \|d_k\|,
	\end{align*}
	and therefore
	\begin{equation*}
		\frac{L_H}{2}\|d_k\|^2 + 2\epsilon_H \|d_k\| - 
		\left\|\nabla f(x_k+d_k)\right\| \; \ge \; 0.
	\end{equation*}
\refer{ By treating the left-hand side as a quadratic in $\|
  d_k \|$, and applying Lemma~\ref{lem:T1} with $a=2$,
    $b=2L_H$ and $t = \| \nabla f(x_k+d_k) \| / \epsH^2$, we obtain
  from this bound that
	\begin{eqnarray} \label{eq:trinomialboundnormdrk}
		\|d_k\| &\ge &\frac{-2\epsilon_H+
		\sqrt{4\epsilon_H^2 + 2 L_H \|\nabla f(x_k+d_k)\|}}{L_H} \nonumber \\
		&= &\frac{-2+
		\sqrt{4 + 2 L_H \|\nabla f(x_k+d_k)\|/\epsH^2}}{L_H}\,  \epsH \nonumber \\
		&\ge &\frac{-2+\sqrt{4 + 2 L_H}}{L_H}\min\left( 
		\|\nabla f(x_k+d_k)\|/\epsH^2, 1\right) \epsH \nonumber \\ 
		&=  &\frac{2 L_H}{L_H(2+\sqrt{4+2 L_H})}
		\min\left( \|\nabla f(x_k+d_k)\| / \epsH,\epsH\right) \nonumber \\ 
		&= &\frac{1}{1+\sqrt{1+L_H/2}}
		\min\left( \|\nabla f(x_k+d_k)\| / \epsH,\epsH\right).
	\end{eqnarray}
Therefore, if the unit step is accepted, we have
\begin{align} 
\nonumber
		f(x_k) & - f(x_k+d_k) \\
\label{eq:regnewtunitacc}
& \geq \frac{\eta}{6}\|d_k\|^3 
		\geq \frac{\eta}{6}\left[\frac{1}{1+\sqrt{1+L_H/2})}\right]^3
		\min\left( \|\nabla f(x_k+d_k)\|^3\epsH^{-3},\epsH^3\right).
\end{align}
}
	
If the unit step does not yield a sufficient decrease, there must be a
value $j \ge 0$ such that \eqref{eq:lsdecrease} is not satisfied for
$\alpha = \theta^j$. For such $j$, and using again \eqref{eq:LH}, we
have
	\begin{align*}
		-\frac{\eta}{6}\theta^{3j}\|d_k\|^3 &\leq 
		f(x_k+\theta^j d_k) -f(x_k) \\
		&\leq \theta^j g_k^\top d_k + \frac{\theta^{2j}}{2}
		d_k^\top \nabla^2 f(x_k) d_k 
		+ \frac{L_H}{6}\theta^{3j} \|d_k\|^3 \\
		&= \theta^j\left(1-\frac{\theta^j}{2}\right)g_k^\top d_k  
		- \epsilon_H \theta^{2j}\|d_k\|^2  +
		\frac{L_H}{6}\theta^{3j}\|d_k\|^3.
		\\
		&\leq - \epsilon_H \theta^{2j}\|d_k\|^2 + 
		\frac{L_H}{6}\theta^{3j}\|d_k\|^3.
	\end{align*}
	Thus, for any $j \geq 0$ for which sufficient decrease is not
        obtained, one has
	\begin{equation} \label{eq:boundnotunitregstep}
		\theta^j \; \geq \; \frac{6}{L_H+\eta}\epsilon_H \|d_k\|^{-1}.
	\end{equation}
	Meanwhile, we have from the definition of $j_r$ that 
	\begin{equation*}
		\theta^{j_r} \le \frac{6}{L_H+\eta} 
		\frac{\epsilon_H^2}{U_g} \le \frac{6}{L_H+\eta} 
		\epsilon_H \frac{\epsilon_H}{U_g} \le \frac{6}{L_H+\eta} 
		\epsilon_H \|d_k\|^{-1},
	\end{equation*}
	using the upper bound~\eqref{eq:upperboundnormdrk}. By comparing this bound 
	with \eqref{eq:boundnotunitregstep}, we deduce that the backtracking 
	line-search procedure terminates with $j_k \le j_r+1$, where $j_k \ge 1$ by 
	our earlier assumption. Thus, since \eqref{eq:boundnotunitregstep} is 
	satisfied for $j=j_k-1$, we have
	\begin{equation*}
		\theta^{j_k} \geq \frac{6\theta}{L_H+\eta}\epsilon_H \|d_k\|^{-1},
	\end{equation*}
	and therefore
	\begin{equation*}
		f(x_k) - f(x_k+\theta^{j_k}d_k) \; \geq \; \frac{\eta}{6}\theta^{3 j_k}
		\|d_k\|^3 \; \geq \; \frac{\eta}{6}
		\left[\frac{6\theta}{L_H+\eta}\right]^{3}\epsilon_H^3.
	\end{equation*}
By combining this bound with~\eqref{eq:regnewtunitacc}, 
obtained for the unit-step case, we obtain the result.
\end{proof}

By combining the estimates of function decrease proved in the lemmas
above, \refer{we} bound the number of iterations needed by
Algorithm~\ref{algo:sorn} to satisfy the approximate second-order
optimality conditions \eqref{eq:optwcc}.

\begin{theorem} \label{theo:itwcc}
Let Assumptions~\ref{assum:compactlevelset} and~\ref{assum:fC22} hold.
Then Algorithm~\ref{algo:sorn} reaches an iterate that
satisfies \eqref{eq:optwcc} in at most
\refer{
\begin{equation} \label{eq:itwcc}
		\mathcal{C}\max\left\{ \epsilon_g^{-3}\epsilon_H^{3},
		\epsilon_g^{-3/2},\epsilon_H^{-3}\right\} \;\; \mbox{\rm iterations},
\end{equation}
}
where
\begin{equation} \label{eq:defcC}
		\mathcal{C} := c^{-1}(f(x_0)-\flow), \quad
		c := \min\left\{c_g,c_v,c_n,c_r\right\}.
\end{equation}
\end{theorem}
\begin{proof}
Suppose $l$ is an iteration at which the conditions for termination
are {\em not} satisfied. We consider in turn the various types of
steps that could have been taken at iteration $l$, and obtain a lower
bound on the amount of decrease obtained from
each. Table~\ref{tab:stepsalgo} is helpful in working through the
various cases.  We consider two main cases, and several subcases.

\medskip

\textbf{Case 1: $\lambda_l < -\epsH$.} 

From Table~\ref{tab:stepsalgo}, we see that in this case, the search
direction is either a scaling of $-g_k$, or the
most-negative-curvature direction $v_k$. When $R_l < -\epsilon_H$, we
have $d_l = \frac{R_l}{\|g_l\|}g_l$, and
Lemma~\ref{lemma:eigsteplength} indicates the following bound on
function decrease:
\[
		f(x_l) - f(x_{l+1}) \; \ge \; c_e \epsilon_H^3.
\]
When $R_l \in [-\epsH,\epsH]$ and $\| g_l \| > \epsg$, we have 
$d_l = -g_l/\| g_l\|^{1/2}$.  Thus, using
Lemma~\ref{lemma:gradsteplength}, we have
\[
		f(x_l) - f(x_{l+1}) \; \ge \; c_g \min\left\{
		\epsilon_g^{3}\epsilon_H^{-3},\epsilon_g^{3/2}\right\}.
\]
For the remaining cases of ``$\|g_l\| \le \epsg$ and 
$R_l \in [-\epsH,\epsH]$" and ``$\|g_l\| > \epsg$ and $R_l > \epsH$", 
the search direction is necessarily $v_l$. We have from
Lemma~\ref{lemma:eigsteplength} that
\[
		f(x_l) - f(x_{l+1}) \; \ge \; c_e 
		\left[\frac{|d_l^\top \nabla^2 f(x_l) d_l|}{\|d_l\|^2}\right]^3 
		 = c_e | \lambda_l|^3 \ge  c_e \epsilon_H^3.
\]

\medskip
	
\textbf{Case 2: $\lambda_l \ge -\epsilon_H$, $\| g_l \| > \epsg$, and 
$\| g_{l+1} \| > \epsg$.}
	
In this case, we have three possible choices for the search
direction. The first one is $d_l = -g_l / \|g_l\|^{1/2}$, in which
case we have from Lemma~\ref{lemma:gradsteplength} that
\[
		f(x_l) - f(x_{l+1}) \; \ge \; c_g \min\left\{
		\epsilon_g^{3}\epsilon_H^{-3},\epsilon_g^{3/2}\right\}.	
\]
The second possible choice is the Newton direction $d_l=d^n_l$.  Using
Lemma~\ref{lemma:newtsteplength}, we obtain
\[
		f(x_l) - f(x_{l+1}) \; \ge \; c_n \min\left\{ 
		\epsilon_g^{3/2},\epsilon_H^3 \right\}.
\]	
The third choice is the regularized Newton direction $d_l=d^r_l$, for
which Lemma~\ref{lemma:regnewtsteplength} yields
\refer{
\[
		f(x_l) - f(x_{l+1}) \; \ge \; c_r \min\left\{
		\|g_{l+1}\|^3 \epsH^{-3},\epsH^3\right\}
		\; \ge \; c_r \min\left\{\epsg^3 \epsH^{-3},\epsH^3\right\}.
\]	
}	

By putting all these bounds together, we obtain the following lower
bound on the decrease in $f$ on iteration $l$:
\refer{
\begin{equation} \label{eq:decreaseepsgh}
	f(x_l) - f(x_{l+1}) \ge c \min\left\{
	\epsilon_g^{3}\epsilon_H^{-3},\epsilon_g^{3/2},
	\epsilon_H^3\right\},
\end{equation}
}
where $c$ is defined in \eqref{eq:defcC}. 
Consequently, summing across all iterations up to $k$ yields
\refer{
\[
		f(x_0) - \flow \ge \sum_{l=0}^{k-1} f(x_l)-f(x_{l+1})
		\ge k c \min\left\{
		\epsilon_g^{3}\epsilon_H^{-3},\epsilon_g^{3/2},
		\epsilon_H^3\right\},
\]
}
which implies that $k$ is bounded above by~\eqref{eq:itwcc}.
Therefore, there must exist a finite index $k_\epsilon$ such
that~\eqref{eq:optwcc} is satisfied. For this index, the
bound~\eqref{eq:itwcc} applies, hence the result.
\end{proof}

\refer{
We now look further into the various components of the bound established in 
Theorem~\ref{theo:itwcc}.}
\paragraph{Dependencies on the tolerances $(\epsg,\epsH)$}
\refer{
The result~\eqref{eq:itwcc} makes explicit the variation of the bound with 
respect to the two tolerances. As this result differs from those in the 
literature, we follow two usual approaches to ease the comparison with other 
methods.\\
Letting $\epsg = \eps$ and $\epsH = \sqrt{\eps}$ for some $\eps \in (0,1)$ 
allows to equate all components of the maximum term in~\eqref{eq:itwcc}; indeed,
\[
\epsilon_g^{-3}\epsilon_H^{3} = \epsilon_g^{-3/2} = \epsilon_H^{-3} =
\eps^{-3/2},
\]
and therefore our bound is $\mathcal{O}(\eps^{-3/2})$. On the other hand, the 
choice $\epsg=\epsH=\eps$, that puts first- and second-order requirement on an 
equal footing, leads to a bound in $\mathcal{O}(\eps^{-3})$. 
Both match the optimal bounds known for second-order
globally convergent methods in terms of iteration count.}

\paragraph{Dependencies on problem-algorithmic constants} 
\refer{Although our main goal 
is to analyze dependencies with respect to the tolerances, our bounds can also 
reflect dependencies on problem-dependent quantities, namely, the initial 
function value discrepancy $f(x_0)-\flow$ and the Lipschitz constants $L_g$ 
and $L_H$. It can be seen from  the lemmas of this subsection that 
\begin{equation*}
	c_e = \mathcal{O}(L_H^{-3}),
	\quad c_g = \mathcal{O}(L_H^{-3/2}),
	\quad c_n = \mathcal{O}(L_H^{-3}), 
	\quad c_r = \mathcal{O}(L_H^{-3}).
\end{equation*}
As a result, the iteration complexity of our method is in
\[
	\mathcal{O}\left( (f(x_0)-\flow)L_H^{3}\max\left\{ 
	\epsilon_g^{-3}\epsilon_H^{3},\epsilon_g^{-3/2},\epsilon_H^{-3}
	\right\}\right).
\]
}

\subsection{Evaluation/Inner Iteration Complexity}
\label{subsec:wccexactevals}

We now discuss the function evaluation complexity of
Algorithm~\ref{algo:sorn}, which counts the number of function calls
required by the algorithm before its termination conditions are
satisfied. We need to refine the iteration complexity analysis of
Section~\ref{subsec:wccexactits} to take into account the function
evaluations associated with the backtracking line-search process.

\begin{theorem} \label{theo:evalwcc}
	Suppose that Assumptions~\ref{assum:compactlevelset} and~\ref{assum:fC22} 
	hold. The number of function evaluations required by 
	Algorithm~\ref{algo:sorn} prior to reaching a point that satisfies 
	\eqref{eq:optwcc} is at most
	\refer{
	\begin{equation} \label{eq:evalwcc}
		\left[1+\mathcal{K} + \log_{\theta}\left(\min\{\epsilon_H^2,
		\epsilon_g^{1/2}\epsilon_H^{-1}\}\right) \right]\mathcal{C}\max
		\left\{ \epsilon_g^{-3}\epsilon_H^{3},\epsilon_g^{-3/2},
		\epsilon_H^{-3}\right\},
	\end{equation}
	}
	where 
	\refer{
	\begin{equation*}
		\mathcal{K} := \left[ \log_{\theta}\left(\min\left\{\frac{3}{L_H+\eta},
		\frac{5}{3},\frac{1}{(L_H+\eta)^{1/2}},
		\sqrt{\frac{3}{(L_H+\eta)U_g}},\frac{6}{(L_H+\eta)U_g}\right\}\right) 
		\right]_+
	\end{equation*}
	}
	and $\mathcal{C}$ is defined as in Theorem~\ref{theo:itwcc}.
\end{theorem}

\begin{proof}
	Theorem~\ref{theo:itwcc} gives a bound on the number of iterations. 
	By Lemmas~\ref{lemma:eigsteplength}--\ref{lemma:regnewtsteplength}, a 
	bound on the corresponding number of function evaluations is
	\begin{equation*}
		\left(1 + \max\left\{j_e,j_g,j_n,j_r\right\}\right)
		\mathcal{C}\max\left\{ \epsilon_g^{-3}\epsilon_H^{3},
		\epsilon_g^{-3/2},\epsilon_H^{-3}\right\}.
	\end{equation*}
	Using the definitions of $j_e$, $j_g$, $j_n$, and $j_r$ from
	Lemmas~\ref{lemma:eigsteplength}, \ref{lemma:gradsteplength},
	\ref{lemma:newtsteplength}, and \ref{lemma:regnewtsteplength}, respectively, 
	and the fact that $\epsilon_g,\epsilon_H \in (0,1)$ yields the result.
\end{proof}

\refer{ With our specific choices of $\epsilon_g$ and
  $\epsilon_H$ mentioned in the previous section, the evaluation
  complexity bounds are
  $\mathcal{O}\left(\log(\frac{1}{\epsilon})\epsilon^{-3/2}\right)$
  and $\mathcal{O}\left(\log(\frac{1}{\epsilon})\epsilon^{-3}\right)$,
  respectively.  We can also derive a bound that includes dependencies
  on problem constants; for instance, the bound corresponding to
  $\epsg=\epsH=\eps$ is
\[
	\mathcal{O}\left(\log\left(\tfrac{\max\{L_H,U_g,U_H\}}{\epsilon}\right)\,
	(f(x_0)-\flow)L_H^{-3}\epsilon^{-3}\right).
\]
}

\subsection{Local Convergence}
\label{subsec:localcv}

In the previous sections, we have derived global complexity guarantees
for Algorithm~\ref{algo:sorn}. We now aim to show rapid local
convergence for the variant of the algorithm that invokes the Local
Phase, Algorithm~\ref{algo:localphase}, rather than terminating as
soon as the conditions \eqref{eq:optwcc} are satisfied.  We note that
local convergence results like the one we prove here have in the past
gone hand-in-hand with global convergence results in smooth nonconvex
optimization (see for example \cite{JNocedal_SJWright_2006}).  More
recently, several works in the optimization literature have 
established rapid local convergence alongside global complexity guarantees
\cite{EGBirgin_JMMartinez_2017,
  CCartis_NIMGould_PhLToint_2011a,FECurtis_DPRobinson_MSamadi_2017a}.

For this section, we will make the following additional assumption.

\begin{assumption} \label{assum:cvlocalmin}
	The sequence of iterates generated by
        Algorithm~\ref{algo:sorn} in conjunction with
        Algorithm~\ref{algo:localphase} converges to a local
        minimizer, that is, a point $x^*$ at which $\nabla f(x^*) = 0$
        and $\nabla^2 f(x^*) \succ 0$.
\end{assumption}

Under this assumption, the following result is immediate.

\begin{lemma} \label{lemma:strongcvxlocal}
Under Assumptions~\ref{assum:compactlevelset}, \ref{assum:fC22}
and~\ref{assum:cvlocalmin}, there exists $k_0 \in \mathbb{N}$ such
that for every $k \ge k_0$, we have for $\mu := \tfrac12 \min \left(
1, \lambda_{\min}\left(\nabla^2 f(x^*)\right) \right) > 0$
that
\begin{equation} \label{eq:strongcvxlocal}
		\mu I \; \preceq \; \nabla^2 f(x_k) \; \preceq \; U_H I,
\end{equation}
	and
\begin{equation} \label{eq:gradsmalllocal}
\|g_k\| < \min\left\{\frac{3 \mu^4}{L_H+\eta}, \epsilon_g\right\}.
\end{equation}
\end{lemma}

Note that the conditions on $k_0$ in Lemma~\ref{lemma:strongcvxlocal}
are such that the combined strategy of
Algorithm~\ref{algo:sorn}-Algorithm~\ref{algo:localphase} will have
entered the Local Phase (Algorithm~\ref{algo:localphase}) before
iteration $k_0$, and will stay in this phase at all subsequent
iterations.

We now establish a local quadratic convergence result.

\begin{theorem} \label{theo:unitsteptaken}
Suppose that Assumptions~\ref{assum:compactlevelset},
\ref{assum:fC22}, and~\ref{assum:cvlocalmin} are satisfied, and let
$\mu$ and $k_0$ be as defined in
Lemma~\ref{lemma:strongcvxlocal}. Then for every $k \ge k_0$, the
method always takes the Newton direction with a unit step length, and we 
have
\begin{equation} \label{theo:newgradlocal}
		\|g_{k+1}\| \leq \frac{L_H}{2\mu^2}\|g_k\|^2  \le \frac38 \| g_k \|.
\end{equation}
\end{theorem}
\begin{proof}
Let $k \geq k_0$, so that we are in the Local Phase
(Algorithm~\ref{algo:localphase}) at iteration $k$. By
Lemma~\ref{lemma:strongcvxlocal}, the Hessian at $\nabla^2 f(x_k)$ is
positive definite, with smallest eigenvalue bounded below by
$\mu>0$. Thus Algorithm~\ref{algo:localphase} computes the Newton
direction $d_k = d^n_k$, and we have
\[
\|d_k\| \le \|g_k\|/ \mu, \quad 	g_k^\top d_k \le -\mu \|g_k\|^2.
\]
We thus have
\begin{align*}
	f(x_k+d_k) - f(x_k) &\leq g_k^\top d_k + \frac12
	d_k^\top \nabla^2 f(x_k) d_k + \frac{L_H}{6}\|d_k\|^3\\ 
	&= \frac12 g_k^\top d_k + \frac{L_H}{6}\|d_k\|^3
	\leq -\frac{\mu}{2}\|g_k\|^2+\frac{L_H}{6}\|d_k\|^3.
\end{align*}
Thus if the sufficient decrease condition $f(x_k+d_k)-f(x_k) \le
-\frac{\eta}{6}\|d_k\|^3$ is {\em not} satisfied for the unit step, we
must have
\begin{equation*}
	\frac{L_H+\eta}{6}\|d_k\|^3 \geq \frac{\mu}{2}\|g_k\|^2,
\end{equation*}
which by the bound $\| d_k \| \le \| g_k \|/\mu$ can be true only if
\[
	\frac{L_H+\eta}{6} \frac{\| g_k \|^3}{\mu^3} \ge \frac{\mu}{2}\|g_k\|^2 \;\;
	\Leftrightarrow \;\; \| g_k \| \ge \frac{3 \mu^4}{L_H+\eta},
\]
which contradicts \eqref{eq:gradsmalllocal}. Thus the unit Newton step
is taken, and we have 
\begin{align*}
	\|g_{k+1}\| = \left\| \nabla f(x_k+d_k)\right\| 
	&= \left\| \nabla f(x_k+d_k) -\nabla f(x_k)-\nabla^2 f(x_k) d_k \right\| \\
	& \le \frac{L_H}{2}\|d_k\|^2 \\
	& \le \frac{L_H}{2\mu^2}\|g_k\|^2 \\
	& < \frac{L_H}{2\mu^2} \frac{3 \mu^4}{L_H+\eta} \|g_k \| \\
	& \le \frac32 \mu^2 \| g_k \| \le \frac38 \|g_k\|,
\end{align*}
completing the proof.
\end{proof}

\section{A Variant with Inexact Directions}
\label{sec:algoinexact}

In Section~\ref{sec:algoexact}, we have assumed that certain
linear-algebra operations in Algorithm~\ref{algo:sorn} --- the linear
system solves of \eqref{eq:fullnewton} and \eqref{eq:regnewton} and
the eigenvalue / eigenvector computation of \eqref{eq:eigenvec} ---
are performed exactly.  In a large-scale setting, the cost of these
operations can be prohibitive, so iterative techniques that perform
these operations {\em inexactly} are of interest. In this section, we
describe inexact methods for these key operations, and examine their
consequences for the complexity analysis.

\subsection{Inexact Eigenvector Calculation: Randomized Lanczos Method}
\label{subsec:inexacteig}

The problem of finding the minimum eigenvalue of the matrix in
\eqref{eq:eigenvec} and its associated eigenvector can be reformulated
as one of finding the maximum eigenvalue and eigenvector of a positive
semidefinite matrix. The Lanczos algorithm with
a random starting vector is an appealing option for the latter
problem, yielding an $\epsilon$-approximate eigenvector in
$\mathcal{O}\left(\log(n/\delta)\epsilon^{-1/2}\right)$ iterations,
with probability at least
$1-\delta$~\cite{JKuczynski_HWozniakowski_1992}. This fact has been
used in several methods that achieve fast convergence
rates~\cite{NAgarwal_ZAllenZhu_BBullins_EHazan_TMa_2017,
YCarmon_JCDuchi_2017,YCarmon_JCDuchi_OHinder_ASidford_2017}. In order to apply 
this method to a matrix that is not positive definite, one must make use of a 
bound on the Hessian norm. For sake of completeness, we spell out the 
procedure in the following lemma. 

\begin{lemma} \label{lemma:Lanczosfixedproba}
Let $H$ be a symmetric matrix satisfying $\|H\| \le M$ for some
$M>0$. Suppose that the Lanczos procedure is applied to find the
largest eigenvalue of $M I -H$ starting at a random vector uniformly
distributed over the unit sphere.  Then, for any $\varepsilon>0$ and
$\delta \in (0,1)$, there is a probability at least $1-\delta$ that the 
procedure outputs a unit vector $v$ such that
\begin{equation} \label{eq:approxeigvec}
	v^\top H v \le \lambda_{\min}(H) + \varepsilon
\end{equation}
in at most
\begin{equation} \label{eq:itsLanczos}
	\min\left\{n,\frac{\ln(n/\delta^2)}{2\sqrt{2}}\sqrt{\frac{M}{\varepsilon}}
	\right\} \quad \mbox{iterations.}
\end{equation}

After at most $n$ iterations, the procedure obtains a unit vector $v$ such
that $v^\top Hv = \lambdamin(H)$, with probability
$1$.
\end{lemma}
\begin{proof}
By definition, the matrix $H'=M I - H$ it is a symmetric positive
semidefinite matrix with its spectrum lying in $[0,2M]$.  Applying the
Lanczos procedure to this matrix from a starting point drawn randomly
from the unit sphere yields a unit vector $v$ such that
\begin{equation} \label{eq:Lanczosvector}
		v^\top H' v \geq \left(1-\frac{\varepsilon}{2M} \right)\lambda_{\max}(H') 
		\geq \left(1-\frac{\varepsilon}{2M}\right)(M-\lambda_{\min}(H))
\end{equation}
in no more than $\min\left\{n,
\frac{\ln(n/\delta^2)}{4\sqrt{\varepsilon/(2M)}}\right\}$ iterations
with probability at least $1-\delta$. (This result  is from~\cite[Theorem
  4.2]{JKuczynski_HWozniakowski_1992} extended by a continuity
argument from the positive definite case to the positive semidefinite
case; see~\cite[Remark 7.5]{JKuczynski_HWozniakowski_1992}.)
Moreover, using~\eqref{eq:Lanczosvector}, we have
\begin{align*}
	v^\top H v &= -v^\top H' v + M \\
	&\le -\left(1-\frac{\varepsilon}{2M}\right)(M-\lambda_{\min}(H)) + M \\
	&= -M + \lambda_{\min}(H) + \frac{\varepsilon}{2} 
	-\frac{\varepsilon}{2M}\lambda_{\min}(H) + M \\
	&= \lambda_{\min}(H) + \frac{\varepsilon}{2} 
	-\frac{\varepsilon}{2M}\lambda_{\min}(H) \\
	&\le \lambda_{\min}(H) + \frac{\varepsilon}{2} 
	+\frac{\varepsilon}{2M}M \\
	&= \lambda_{\min}(H) + \varepsilon,
\end{align*}
as required. 
\end{proof}

Lemma~\ref{lemma:Lanczosfixedproba} admits the following variant, for
the case in which we fix the number of Lanczos iterations.

\begin{lemma} \label{lemma:Lanczosfixedits}
Let $H$ be a symmetric matrix with $\|H\| \le
M$. Suppose that $q$ iterations of the Lanczos procedure are applied
to find the largest eigenvalue of $M I -H$ starting at a random vector
uniformly distributed over the unit sphere.  Then for any
$\varepsilon>0$, the procedure outputs a unit vector $v$
such that $v^\top H v \le \lambda_{\min}(H) + \varepsilon$ with
probability at least
\begin{equation} \label{eq:probaLanczos}
		1-\delta \; = \; 1-
		\sqrt{n}\exp\left[-\sqrt{2}q \sqrt{\frac{\varepsilon}{M}} 
		\right].
\end{equation}	
\end{lemma}

We point out that the choice $\delta=0$ (or, equivalently, $q=n$) is possible, 
that is, after $n$ iterations, the Lanczos procedure started with a random 
vector uniformly generated over the unit sphere returns an approximate 
eigenvector with probability 
one~\cite[Theorem 4.2 (a)]{JKuczynski_HWozniakowski_1992}.

\subsection{Inexact Newton and Regularized Newton Directions: Conjugate 
Gradient Method}
\label{subsec:inexactnewt}

Here we describe the use of the conjugate gradient (CG) algorithm to
solve the symmetric positive definite linear systems
~\eqref{eq:fullnewton} or~\eqref{eq:regnewton} --- the Newton and
regularized Newton equations, respectively. The conjugate gradient
method is the most popular iterative method for positive definite
linear systems, due to its rich convergence theory and strong
practical performance.  It has also been popular in the context of
nonconvex smooth minimization; see \cite{TSteihaug_1983}.  It requires only
matrix-vector operations involving the coefficient matrix (often these
can be found or approximated without explicit knowledge of the matrix)
together with some vector operations. It does not require knowledge or
estimation of the extreme eigenvalues of the matrix.

We apply CG to a system $Hd=-g$ where there are positive quantities
$m$ and $M$ such that $mI \preceq H \preceq MI$, so that the condition
number $\kappa$ of $H$ is bounded above by $M/m$.  Standard
convergence theory indicates that CG outputs a vector $d$
such that $\|H y + g\| \le \zeta \|g\|$ (for $\zeta \in (0,1)$) in
\begin{equation*}
	\mathcal{O}\left(\min\left\{n,
	\kappa^{1/2}\log(\kappa/\zeta)\right\} \right) \quad
	\mbox{iterations,}
\end{equation*}
with $\kappa$ being the condition number of $H$ (we obtain the result as a 
Corollary of Lemma~\ref{lemma:itsCGapprox} below). We use a different stopping 
criterion, namely
\begin{equation} \label{eq:stopcritCG}
	\left\| H d + g \right\| \le \frac12 \zeta \min\left\{ \|g\|,
	m\|d\|\right\}
\end{equation}
for some $\zeta \in (0,1)$.  This criterion is stronger than the one
typically used in truncated Newton-Krylov methods,  in that we require
the residual norm to be bounded by a multiple of the norm of the
approximate direction, as well as being bounded by a specified
fraction of the initial residual norm. The extra criterion resembles
the so-called \emph{s-condition} arising in cubic regularization
techniques, where the approximate minimizer $s_k$ of the cubic model
$m_k$ is required to satisfy
\begin{equation} \label{eq:cubicmodelcond}
	\left\| \nabla m_k(s_k) \right\| \le 
	\mathcal{O}\left(\|s_k\|^2\right).
\end{equation}
This property provides a lower bound on $\|s_k\|$, that is instrumental	
in obtaining the optimal complexity order of 
$\mathcal{O}(\epsilon_g^{-3/2})$ for first-order
convergence~\cite{CCartis_NIMGould_PhLToint_2011b}. Our condition 
replaces $\|s_k\|^2$ by $m\|d_k\|$, but serves a similar purpose.

The next lemma establishes a bound on the number of CG iterations needed 
to reach the desired accuracy.

\begin{lemma} \label{lemma:itsCGapprox}
Let $H d = -g$ be a linear system with $H$ symmetric and 
$m I \preceq H \preceq M I$, where $m \in (0,1)$, $M > 0$, and $\|g\| > 0$. 
Then the conjugate gradient algorithm computes a vector $d$ such 
that~\eqref{eq:stopcritCG} holds for some 
$\zeta \in (0,1)$ in at most 
\begin{equation} \label{eq:itsCGapprox}
	\min\left\{ n, \tfrac12 \sqrt{\kappa}
	\ln\left(4\refer{\kappa^{3/2}}/\zeta\right)\right\},
\end{equation}
iterations, where $\kappa = M/m$.
\end{lemma}
\begin{proof}
Let $d^{(q)}$ be the iterate obtained at the $q$-th iteration of the 
conjugate gradient method applied to $Hd=-g$, with $d^{(0)}=0$. The classical 
bound on the behavior of the conjugate gradient 
residual~\cite[Section~5.1]{JNocedal_SJWright_2006} yields
\begin{equation} \label{eq:boundCGresAnorm}
	\left\|d^{(q)}+H^{-1}g\right\|_H \le 2\left(
	\frac{\sqrt{\kappa}-1}{\sqrt{\kappa}+1}\right)^q \|H^{-1}g\|_H,
\end{equation}		
where $\|x\|_H=\sqrt{x^\top H x}$. From this definition and the bounds on the 
spectrum of $H$, we have
\refer{
\begin{align*}
\| d^{(q)} + H^{-1} g \|_H^2 & = (d^{(q)}+H^{-1}g)^TH(d^{(q)}+H^{-1}g)  \\
& = (Hd^{(q)}+g)^TH^{-1} (Hd^{(q)}+g) \ge \frac{1}{M} \|Hd^{(q)}+g\|^2.
\end{align*}
as well as
\[
	\| H^{-1}g \|_H^2  = g^T H^{-1} g \le \frac{1}{m} \|g\|^2.
\]
}
By substituting these bounds into \eqref{eq:boundCGresAnorm}, we
obtain the following relation:
\begin{equation} \label{eq:boundCGres}
	\left\| H d^{(q)} + g \right\| \le 2 \refer{\kappa^{1/2}}\left(
	\frac{\sqrt{\kappa}-1}{\sqrt{\kappa}+1}\right)^q \|g\|.
\end{equation}
Thus, as long as our stopping criterion is not satisfied, we have
\begin{equation} \label{eq:contradCGres}
	\frac12 {\zeta}\min\left\{\|g\|,m\|d^{(q)}\|\right\} \le 
	2\refer{\kappa^{1/2}}
	\left(\frac{\sqrt{\kappa}-1}{\sqrt{\kappa}+1}\right)^q \|g\|.
\end{equation}
Furthermore, defining $r^{(q)} = H d^{(q)}+g$, we have
\begin{equation*}
	\|d^{(q)}\| = \|H^{-1}(-g+r^{(q)})\| \ge \frac{\|g-r^{(q)}\|}{M} 
	\ge  \frac{\sqrt{\| g\|^2  - 2g^T r^{(q)} + \| r^{(q)} \|^2}}{M} \ge
	\frac{\|g\|}{M},
\end{equation*}
for all $q \ge 1$, where we used the fact that using the facts that
$r^{(0)} = g$ and that in CG, the residuals are orthogonal:
$(r^{(i)})^Tr^{(j)}=0$ for $i \neq j$. Using this bound
within~\eqref{eq:contradCGres}, we obtain
\[
	\frac{\zeta}{2}\min\{1,{m}/{M}\}\|g\| \le 
	2 \refer{\kappa^{1/2}} \left(
	\frac{\sqrt{\kappa}-1}{\sqrt{\kappa}+1}\right)^q \|g\| \;\; \Leftrightarrow \;\;
	\frac{\zeta}{\refer{4\kappa^{3/2}}} \le \left(
	\frac{\sqrt{\kappa}-1}{\sqrt{\kappa}+1}\right)^q.
\]
By taking logarithms on both sides, we arrive at
\begin{equation*}
	q \le \frac{\ln(\zeta/(\refer{4\kappa^{3/2}}))}{\ln\left(
	\frac{\sqrt{\kappa}-1}{\sqrt{\kappa}+1}\right)} = 
	\frac{\ln(\refer{4\kappa^{3/2}}/\zeta)}{\ln\left(1 + 
	\frac{2}{\sqrt{\kappa}-1}\right)} \le \frac12 \sqrt{\kappa}
	\ln\left(\frac{\refer{4\kappa^{3/2}}}{\zeta}\right),
\end{equation*}
where the bound $\ln(1+\frac{1}{t}) \ge \frac{1}{t+1/2}$ was used to obtain 
the last inequality.
\end{proof}

\subsection{Complexity Analysis Based on Inexact Computations}
\label{subsec:wccinexact}

We present a variant of our main algorithm, specified as
Algorithm~\ref{algo:inexsorn}, in which computation of approximate
eigenvectors and linear system solves are performed inexactly by the
means described above. Algorithm~\ref{algo:inexsorn} requires two parameters 
not used in Algorithm~\ref{algo:sorn}: the upper bound \refer{$U_H$} on the 
Hessian norms, defined in \eqref{eq:boundscompact}, and a probability 
threshold $\delta$. As we expect only to recover inexact global complexity 
guarantees, the method does not exploit a local phase.

When Algorithm~\ref{algo:inexsorn} terminates, condition~\eqref{eq:optwcc} 
must hold.  At termination, we have $\min(\|g_k\|, \| g_{k+1} \|) \le \epsg$ 
and $\lambda^i_k \ge -\tfrac12 \epsH$. With high probability, $\lambda^i_k$ is 
within $\tfrac12 \epsH$ of $\lambdamin(\nabla^2 f(x_k))$, so we must have 
$\lambdamin(\nabla^2 f(x_k)) \ge -\epsH$, thus satisfying \eqref{eq:optwcc}.

\begin{algorithm}[H]
\caption{Inexact Second-Order Line Search Method}
\label{algo:inexsorn}
\begin{algorithmic}
\STATE \emph{Init.} Choose $x^0 \in \R^n$, $\theta \in (0,1)$, 
$\zeta,\delta \in [0,1)$, $\eta > 0$, $\epsg \in (0,1)$, $\epsH \in (0,1)$, 
$U_H>0$ satisfying \eqref{eq:boundscompact};
\FOR{$k=0,1,2,\dotsc$}
\STATE \textbf{Step 1. (First-Order)} Set
$g_k = \nabla f(x_k)$;
\IF{$\|g_k\|=0$}
\STATE Go to Step 2;
\ENDIF
\STATE Compute $R_k = \frac{g_k^\top \nabla^2 f(x_k) g_k}{\|g_k\|^2}$;
\IF{$R_k<-\epsH$}
\STATE{Set $d_k = \frac{R_k}{\|g_k\|}g_k$ and go to Step LS;}
\ELSIF{$R_k \in [-\epsH,\epsilon_H]$ and  $\|g_k\| > \epsilon_g$}
\STATE{Set $d_k = -\frac{g_k}{\|g_k\|^{1/2}}$ and go to Step LS;}
\ELSE
\STATE{Go to Step 2;}
\ENDIF
\STATE \textbf{Step 2. (Inexact Second-Order)} Compute an inexact 
eigenvector $v^i_k$ such that (with probability $1-\delta$)
\begin{equation} \label{eq:inexeigenvec}
	[v^i_k]^\top g_k \le 0,\quad \|v^i_k\|=|\lambda^i_k|,\quad \lambda_k^i \le 
	\lambda_{\min}\left(\nabla^2 f(x_k)\right) + \frac{\epsilon_H}{2},
\end{equation}
where $\lambda_k^i = [v^i_k]^\top \nabla^2 f(x_k) v^i_k/\|v^i_k\|^2$;
\IF{$\| g_k \| \le \epsg$ and $\lambda^i_k \ge -\tfrac12 \epsH$}
\STATE{\textbf{Terminate};}
\ELSIF{$\lambda^i_k < -\tfrac12 \epsH$}
\STATE{\textbf{(Negative Curvature)} Set $d_k=v^i_k$;}
\ELSIF{$\lambda^i_k > \tfrac32 \epsH$} 
\STATE \textbf{(Inexact Newton)} Use conjugate gradient to calculate 
$d_k=d^{in}_k$, where
\begin{equation} \label{eq:approxnewtonCG}
\| \nabla^2 f(x_k)d^{in}_k+g_k \| \leq 
\frac{\zeta}{2}\min\left\{\|g_k\|,\epsilon_H \|d^{in}_k\| \right\};
\end{equation}
\ELSE
\STATE  \textbf{(Inexact regularized Newton)} Use conjugate gradient to 
calculate $d_k=d^{ir}_k$, where
\begin{equation} \label{eq:approxregnewtonCG}
	\|(\nabla^2 f(x_k)+2\epsilon_H I)d^{ir}_k+g_k \| \leq 
	\frac{\zeta}{2}\min\left\{\|g_k\|,\epsilon_H \|d^{ir}_k\|\right\};
\end{equation}
\ENDIF
\STATE Go to Step LS;
\STATE \textbf{Step LS. (Line Search)} Compute a step length 
$\alpha_k=\theta^{j_k}$, where $j_k$ is the smallest nonnegative integer such 
that
\begin{equation} \label{eq:lsdecrease2}
	f(x_k + \alpha_k d_k) < f(x_k) - \frac{\eta}{6}\alpha_k^3 \|d_k\|^{3}
\end{equation}
holds, and set $x_{k+1} = x_k+\alpha_k d_k$;
\IF{$d_k=d^{in}_k$ or $d_k=d^{ir}_k$ and $\| \nabla f(x_{k+1}) \| \le \epsg$}
\STATE{\textbf{Terminate}};
\ENDIF
\ENDFOR
\end{algorithmic}
\end{algorithm}

\begin{table}[ht!]
	\begin{center}
	\begin{tabular}{|l|l|c||ll|}
		\hline
		\multicolumn{3}{|c||}{Context} &Direction &Decrease \\
		\hline
		$\|g_k\| = 0$ &- &$\lambda^i_k<-\frac{\epsH}{2}$ 	&$v^i_k$ 
		&Lemma~\ref{lemma:eigsteplength} \\ \hline
		& $R_k < -\epsH$ & & $R_k g_k/\|g_k\|$ & Lemma~\ref{lemma:eigsteplength}\\ 
		\hline
		$\| g_k \| > \epsg$ & $R_k \in [-\epsH,\epsH]$ & & $-g_k/\|g_k\|^{1/2}$ 
		& Lemma~\ref{lemma:gradsteplength}\\ \hline
		$\| g_k \| \le \epsg$ & $R_k \in [-\epsH,\epsH]$ 
		& $\lambda^i_k<-\tfrac12 {\epsH}$ & $v^i_k$ 
		& Lemma~\ref{lemma:eigsteplength}\\
		$\| g_k \| \le \epsg$ & $R_k \in [-\epsH,\epsH]$ 
		& $\lambda^i_k \in [-\tfrac12 {\epsH},\tfrac32 \epsH]$ & $d^{ir}_k$ 
		& Lemma~\ref{lemma:inxactregnewtsteplength}\\ 
		$\| g_k \| \le \epsg$ & $R_k \in [-\epsH,\epsH]$ 
		& $\lambda^i_k > \tfrac32 \epsH$ & $d^{in}_k$ 
		& Lemma~\ref{lemma:inxactnewtsteplength} \\ \hline
		$\| g_k \| > \epsg$ & $R_k > \epsH$ & $\lambda^i_k<-\tfrac12 {\epsH}$ 
		& $v^i_k$ & Lemma~\ref{lemma:eigsteplength}\\
		$\| g_k \| > \epsg$ & $R_k > \epsH$ 
		& $\lambda^i_k \in [-\tfrac12 {\epsH},\tfrac32 \epsH]$ & $d^{ir}_k$ 
		& Lemma~\ref{lemma:inxactregnewtsteplength}\\
		$\| g_k \| > \epsg$ & $R_k > \epsH$ & $\lambda^i_k >\tfrac32 \epsH$ 
		& $d^{in}_k$ & Lemma~\ref{lemma:inxactnewtsteplength}\\ \hline
	\end{tabular}
	\end{center}
	\vspace*{0.5ex}
	\label{tab:stepsinexalgo}
	\caption{Steps and associated decrease lemmas for 
	Algorithm~\ref{algo:inexsorn}.}
\end{table}

Table~\ref{tab:stepsinexalgo} shows a summary of the possible choices
for the search direction. It shows the same number of cases as 
Table~\ref{tab:stepsalgo}, with the context now determined by the
eigenvalue estimate $\lambda_k^i$, with one exception. There is an
extra row for the case $\| g_k \| \le \epsg$, $R_k \in
[-\epsH,\epsH]$, $\lambda^i_k > \tfrac32 \epsH$, because of possible
(but low-probability) failure of the randomized Lanczos process to
detect the smallest eigenvalue of $\nabla^2 f(x_k)$ to the required
accuracy.
Table~\ref{tab:stepsinexalgo} mentions two additional lemmas, that
respectively replace Lemmas~\ref{lemma:newtsteplength}
and~\ref{lemma:regnewtsteplength} in order to take inexactness into
account. We state and prove these results next.

\begin{lemma} \label{lemma:inxactnewtsteplength}
Let Assumptions~\ref{assum:compactlevelset} and~\ref{assum:fC22} hold.
Suppose that an inexact Newton direction $d_k=d^{in}_k$ is computed at
the $k$-th iteration of Algorithm~\ref{algo:inexsorn}. Then with
probability at least $1-\delta$, the backtracking line search
terminates with step length $\alpha_k = \theta^{j_k}$, with $j_k \le
\jin +1$, where
\begin{equation} \label{eq:inexactnewtlsits}	
	\jin  :=  \left[ \frac12 \log_{\theta}\left(\frac{3}{L_H+\eta}\,
	\frac{(1-\zeta)\epsilon_H^2}{U_g\sqrt{1+\zeta^2/4}}\right) \right]_+,
\end{equation}
and we have 
\refer{
\begin{equation} \label{eq:inexactnewtsteplength}
	f(x_k)-f(x_k+\alpha_k d_k) \; \geq \; \cin
	\min\left\{\|\nabla f(x_k+\alpha_k d_k)\|^3\epsH^{-3},
	\epsH^3\right\},
\end{equation}
where
\begin{equation*}
	\cin := \frac{\eta}{6}\min\left\{ 
	\left[ \frac{4}{\zeta+\sqrt{\zeta^2 + 8 L_H}} \right]^3,
	\left[ \frac{3\theta^2(1-\zeta)}{L_H+\eta}\right]^3\right\}.
\end{equation*}
}
\end{lemma}
\begin{proof}
We observe first that when the Newton step is computed in
Algorithm~\ref{algo:inexsorn}, we have from~\eqref{eq:approxeigvec}
that
\begin{equation*}
	\frac{3\epsilon_H}{2} < \lambda_k^i \le 
	\lambda_{\min}\left(\nabla^2 f(x_k)\right) + \frac{\epsilon_H}{2}
	\; \Rightarrow \;
	\lambda_{\min}\left(\nabla^2 f(x_k)\right) \ge \epsilon_H,
\end{equation*}
with probability $1-\delta$. Suppose first that the step
length $\alpha_k=1$ satisfies the decrease
condition~\eqref{eq:lsdecrease2}. Then, defining 
\begin{equation} \label{eq:def.rk}
	r_k := \nabla^2 f(x_k)d_k + g_k
\end{equation} 
and using the inexactness criterion for the inexact
Newton step $d_k$, we find that the gradient at the next point
$x_k+d_k$ satisfies
\begin{align*}
	\left\| \nabla f(x_k+ d_k) \right\| &=
	\left\| \nabla f(x_k+d_k) - \nabla f(x_k) + \nabla f(x_k) \right\| \\
	&=\left\| \nabla f(x_k+d_k)-\nabla f(x_k)-\nabla^2 f(x_k)d_k + r_k\right\| \\
	&\leq \frac{L_H}{2}\|d_k\|^2 + \|r_k\| 
	\leq \frac{L_H}{2}\|d_k\|^2 + \frac{\zeta}{2}\epsilon_H\|d_k\|.
\end{align*}
\refer{
We obtain a lower bound on $\|d_k\|$ by taking the root of the above quadratic 
and applying Lemma~\ref{lem:T1} with $a=\zeta \epsH/2$, 
$b=2L_H \epsH^2$, and $t=\| \nabla f(x_k+d_k)\|/\epsH^2$ to obtain
\begin{align} 
	\|d_k\| &\ge  \frac{-\tfrac{\zeta}{2}\epsH + 
	\sqrt{\tfrac{\zeta^2}{4}\epsH^2+2L_H \|\nabla f(x_k+d_k)\|}}{L_H} \nonumber\\
	& \ge \frac{-\tfrac{\zeta}{2}\epsH + 
	\sqrt{\tfrac{\zeta^2}{4}\epsH^2+2L_H \epsH^2}}{L_H} \min 
	\left(\|\nabla f(x_k+d_k)\|/\epsH^2,1 \right) \nonumber\\
	& = \frac{-\zeta + \sqrt{\zeta^2+8L_H }}{2L_H} \min 
	\left(\|\nabla f(x_k+d_k)\|/\epsH,\epsH \right) \nonumber\\
	\label{eq:trinomialboundnorminexactdnk}
	&= \frac{4}{\zeta+\sqrt{\zeta^2+8 L_H}}
	\min\left( \|\nabla f(x_k+d_k)\|/\epsH,\epsH\right).
\end{align}
}
Therefore, taking the inexact Newton step with a unit step length guarantees
\refer{
\[
	f(x_k) - f(x_k+d_k) \geq \frac{\eta}{6}\|d_k\|^3 
	\geq \frac{\eta}{6}\left[\frac{4}{\zeta+\sqrt{\zeta^2+8 L_H}}\right]^3
	\min\left( \|\nabla f(x_k+d_k)\|^3\epsH^{-3},\epsH^3\right),
\]
}
so the inequality \eqref{eq:inexactnewtsteplength} is satisfied in the
case of a unit step $\alpha_k=1$. 
	
To complete the proof, consider the case in which the unit step length
does not lead to sufficient decrease. In that case, for any value $j
\geq 0$ such that \eqref{eq:lsdecrease2} is not satisfied, we have 
\begin{align*}
	-\frac{\eta}{6}\theta^{3j}\|d_k\|^3 &\leq 
	f(x_k+\theta^j d_k) -f(x_k) \\
	&\leq \theta^j g_k^\top d_k + \frac{\theta^{2j}}{2}
	d_k^\top \nabla^2 f(x_k) d_k 
	+ \frac{L_H}{6}\theta^{3j} \|d_k\|^3 \\
	&\leq \theta^j (-\nabla^2 f(x_k)d_k+r_k)^\top d_k + \frac{\theta^{2j}}{2}
	d_k^\top \nabla^2 f(x_k) d_k 
	+ \frac{L_H}{6}\theta^{3j} \|d_k\|^3 \\
	&= -\theta^j\left(1-\frac{\theta^j}{2}\right)d_k^\top \nabla^2 f(x_k) d_k  
	+\theta^j d_k^\top r_k + \frac{L_H}{6}\theta^{3j}\|d_k\|^3 \\
	&\leq - \frac{\theta^j}{2}\epsilon_H\|d_k\|^2 + \theta^j \|d_k\|\|r_k\| +
	\frac{L_H}{6}\theta^{3j}\|d_k\|^3 \\
	&\leq - \frac{\theta^j}{2}(1-\zeta)\epsilon_H \|d_k\|^2 + 
	\frac{L_H}{6}\theta^{3j}\|d_k\|^3.
\end{align*}
Thus, for any $j \geq 0$ for which sufficient decrease is not
obtained, we have
\begin{equation} \label{eq:boundnotunitinexnewtstep}
	\theta^{2j} \; \geq \; \frac{3}{L_H+\eta}(1-\zeta)
	\epsilon_H \|d_k\|^{-1}.
\end{equation}
In particular, since~\eqref{eq:boundnotunitinexnewtstep} holds
for $j=0$, we have
\begin{equation} \label{eq:lbnorminexnewtnotunit}
	\|d_k\| \; \ge \; \frac{3}{L_H+\eta}(1-\zeta) \epsH.
\end{equation}
By the definitions of $d_k$ and of $r_k$ in \eqref{eq:def.rk}, we also
have the following upper bound on its norm:
\begin{align*}
	\|d_k\| = \left\|\nabla^2 f(x_k)^{-1}\left(g_k-r_k\right)\right\| 
	\le \|\nabla^2 f(x_k)^{-1}\| \left\|g_k-r_k\right\| 
	&\le \frac{1}{\epsilon_H}\sqrt{\|g_k\|^2 + \|r_k\|^2} \\
	&\le \frac{\sqrt{1+\zeta^2/4}}{\epsilon_H}\|g_k\| \\
	&\le \frac{\sqrt{1+\zeta^2/4}}{\epsilon_H}U_g,
\end{align*}		
using again the fact that $g_k$ and $r_k$ are orthogonal (by the
properties of the CG algorithm), as well as the
criterion~\eqref{eq:approxnewtonCG} and the bound
\eqref{eq:boundscompact}.
	
Meanwhile, for any $j > \jin$, we have
\begin{align*}
	\theta^{2j} < \theta^{2 \jin} \le \frac{3}{L_H+\eta}(1-\zeta)
	\frac{\epsilon_H^2}{U_g\sqrt{1+\zeta^2/4}} 
	&\le \frac{3}{L_H+\eta} (1-\zeta)\epsilon_H 
	\frac{\epsilon_H}{U_g\sqrt{1+\zeta^2/4}} \\
	&\le \frac{3}{L_H+\eta} (1-\zeta)\epsilon_H \|d_k\|^{-1}.
\end{align*}
As a result, \eqref{eq:boundnotunitinexnewtstep} is violated for $j >
\jin$, which means that the line search must terminate with a step
length $\alpha_k = \theta^{j_k}$ satisfying \eqref{eq:lsdecrease2}, with 
$1 \le j_k \le \jin+1$. Since the index $j = j_k-1 \ge 0$ 
satisfies~\eqref{eq:boundnotunitinexnewtstep}, we have
\begin{equation} \label{eq:7yg}
	\theta^{j_k} \geq \sqrt{\frac{3\theta^2}{L_H+\eta}(1-\zeta)}
	\epsilon_H^{1/2} \|d_k\|^{-1/2},
\end{equation}
and from the sufficient decrease condition, we have
\begin{align*}
	f(x_k) - f(x_k+\theta^{j_k}d_k) \geq  \frac{\eta}{6}\theta^{3 j_k}
	\|d_k\|^3 & \ge  \frac{\eta}{6}
	\left[\frac{3\theta^2}{L_H+\eta}(1-\zeta)\right]^{3/2}\epsilon_H^{3/2} 
	\| d_k \|^{3/2} \\
	& \ge  \frac{\eta}{6}
	\left[\frac{3\theta^2}{L_H+\eta}(1-\zeta)\right]^3\epsilon_H^3,
\end{align*}
where the second inequality follows from \eqref{eq:7yg} and the third
inequality follows from \eqref{eq:lbnorminexnewtnotunit} (using the
fact that $\theta \in (0,1)$).  Hence, the claim
\eqref{eq:inexactnewtsteplength} is satisfied in the case of non-unit
step length $\alpha_k$ too, and the proof is complete.
\end{proof}

\begin{lemma} \label{lemma:inxactregnewtsteplength}
Let Assumptions~\ref{assum:compactlevelset} and~\ref{assum:fC22} hold.
Suppose that an inexact regularized Newton direction $d_k=d^{ir}_k$ is
computed at the $k$-th iteration of
Algorithm~\ref{algo:inexsorn}. Then with probability at least
$1-\delta$, the backtracking line search terminates with step length
$\alpha_k = \theta^{j_k}$, with $j_k \le \jir +1$, where
$\jir$ is defined as in~\eqref{eq:inexactnewtlsits}, 
and we have
\refer{
\begin{equation} \label{eq:inexactregnewtsteplength}
	f(x_k)-f(x_k+\alpha_k d_k) \; \geq \; \cir
	\min\left\{\|\nabla f(x_k+\alpha_k d_k)\|^3\epsH^{-3},
	\epsH^3\right\}.
\end{equation}
where 
\begin{equation*}
	\cir := \frac{\eta}{6}\min\left\{ 
	\left[\frac{4}{4+\zeta+\sqrt{(4+\zeta)^2+ 8 L_H}}\right]^3,
	\refer{	\left[\frac{3\theta^2(1-\zeta)}{L_H+\eta}\right]^3} \right\}.
\end{equation*}
}
\end{lemma}
\begin{proof}
The inexact regularized Newton step is computed only when $-\tfrac12
\epsH \le \lambda^i_k \le \tfrac32 \epsH$, so from
\eqref{eq:approxeigvec} with $\varepsilon = \frac12 \epsH$, we have
\begin{equation} \label{eq:8gh}
	\lambda_{\min}(\nabla^2 f(x_k)) + 2 \epsilon_H \ge 
	\lambda^i_k - \tfrac12 \epsilon_H + 2\epsilon_H \ge \epsilon_H,
\end{equation}
with probability at least $1-\delta$.  Suppose first that the step
length $\alpha_k=1$ satisfies the decrease
condition~\eqref{eq:lsdecrease2}. Then, defining $r_k = (\nabla^2
f(x_k) + 2 \epsH I) d_k + g_k$, we
have that
\begin{align*}
	\left\| \nabla f(x_k+ d_k) \right\| &=
	\left\| \nabla f(x_k+d_k) - \nabla f(x_k) + \nabla f(x_k) \right\| \\
	&=\left\| \nabla f(x_k+d_k)-\nabla f(x_k)-\nabla^2 f(x_k)d_k 
	-2\epsilon_H d_k + r_k\right\| \\
	&\leq \frac{L_H}{2}\|d_k\|^2 + 2\epsilon_H \|d_k\| + \|r_k\| \\
	&\leq \frac{L_H}{2}\|d_k\|^2 + \frac{4+\zeta}{2}\epsilon_H\|d_k\|.
\end{align*}
\refer{Reasoning as in \eqref{eq:trinomialboundnorminexactdnk}, with 
$\tfrac{4+\zeta}{2}$ replacing $\tfrac{\zeta}{2}$, we obtain the following 
lower bound on $\|d_k\|$:
\begin{equation} \label{eq:trinomialboundnorminexactdrk}
	\|d_k\| \; \ge \; \frac{4}{4+\zeta+\sqrt{(4+\zeta)^2+8 L_H}}
	\min\left(\|\nabla f(x_k+d_k)\| \epsH^{-1},\epsH\right).
\end{equation}
}
Therefore, taking the unit regularized Newton step guarantees
\refer{
\begin{align*}
	f(x_k) - f(x_k+d_k) &\geq \frac{\eta}{6}\|d_k\|^3 \\
	&\geq \frac{\eta}{6}\left[\frac{4}{4+\zeta+\sqrt{(4+\zeta)^2+8 L_H}}\right]^3
	\min\left(\|\nabla f(x_k+d_k)\|^3 \epsH^{-3},\epsH^3\right),
\end{align*}
}
so the result of the theorem holds in the case in which the unit step
satisfies the sufficient decrease condition.
	
To complete the proof, we consider the case in which $\alpha_k<1$.  In
that case, for any value $j \geq 0$ such that \eqref{eq:lsdecrease2}
is not satisfied, we have from the definition of $r_k$, the bound on
$\|r_k\|$ in the definition of $d^{ir}_k$,  and \eqref{eq:8gh} that
\refer{
\begin{align*}
	-\frac{\eta}{6}\theta^{3j}\|d_k\|^3 &\le f(x_k+\theta^j d_k) -f(x_k) \\
	& \le \theta^j g_k^\top d_k + \frac{\theta^{2j}}{2} 
	d_k^\top \nabla^2 f(x_k) d_k 
	+ \frac{L_H}{6}\theta^{3j} \|d_k\|^3 \\
	& = -\theta^j \left[ \nabla^2 f(x_k) d_k + 2 \epsH d_k - r_k \right]^\top d_k 
	+ \frac{\theta^{2j}}{2} d_k^\top \nabla^2 f(x_k) d_k \\
	& \quad\quad\quad  + \frac{L_H}{6}\theta^{3j} \|d_k\|^3 \\
	&= -\theta^j \left(1-\frac{\theta^j}{2} \right) d_k^\top
	[\nabla^2 f(x_k)+ 2 \epsH I] d_k  -  \theta^{2j} \epsH  \|d_k \|^2 + 
	\theta^j r_k^\top d_k 	\\
	& \quad\quad\quad + \frac{L_H}{6}\theta^{3j} \|d_k\|^3 \\
	& \le  -\frac{\theta^j}{2} \epsH \|d_k \|^2 - \theta^{2j} \epsH \|d_k \|^2 
	+ \theta^j \| r_k \| \|d_k \| + \frac{L_H}{6}\theta^{3j} \|d_k\|^3 \\
	& \le -\frac{\theta^j}{2} \epsH \|d_k \|^2 +
	\theta^j \frac{\zeta}{2} \epsH \|d_k \|^2 
	+ \frac{L_H}{6}\theta^{3j} \|d_k\|^3 \\
	& = -\frac{\theta^j}{2} (1-\zeta)  \epsH \|d_k \|^2 
	+ \frac{L_H}{6}\theta^{3j} \|d_k\|^3.
\end{align*}
}
Thus, for any $j \geq 0$ for which sufficient decrease is not
obtained, one has
\refer{
\begin{equation} \label{eq:boundnotunitinexregnewtstep}
	\theta^{2j} \; \geq \; \frac{3}{L_H+\eta}(1-\zeta)\epsilon_H \|d_k\|^{-1}.
\end{equation}
In particular, setting $j=0$ in this expression, we obtain
\begin{equation} \label{eq:64A}
	\| d_k \| \ge \frac{3}{L_H+\eta} (1-\zeta) \epsH.
\end{equation}
} 
The right-hand side of \eqref{eq:boundnotunitinexregnewtstep} is bounded 
below, since
\begin{align}
	\nonumber
	\|d_k\| = \left\|[\nabla^2 f(x_k)+2\epsilon_H I]^{-1}
	\left(-g_k+r_k\right)\right\| 
	&\le \|[\nabla^2 f(x_k)+2\epsilon_H]^{-1}\| \left\|-g_k+r_k\right\| \\
	\nonumber
	&\le \frac{1}{\epsilon_H}\sqrt{\|g_k\|^2 + \|r_k\|^2} \\
	\nonumber
	&\le \frac{\sqrt{1+\zeta^2/4}}{\epsilon_H}\|g_k\| \\
	\label{eq:6rs}
	&\le \frac{\sqrt{1+\zeta^2/4}}{\epsilon_H}U_g,
\end{align}		
where we used again the orthogonality of $g_k$ and $r_k$ (from the
properties of conjugate gradient) as well as the
condition~\eqref{eq:approxregnewtonCG}.
For any $j > \jir$, we have
\refer{
\[
	\theta^{2j} < \theta^{2\jir} \le \frac{3(1-\zeta)}{L_H+\eta}
	\frac{\epsilon_H^2}{U_g\sqrt{1+\zeta^2/4}}  \le 
	\frac{3(1-\zeta)}{L_H+\eta}\epsilon_H \|d_k\|^{-1},
\]
} where the last inequality follows from \eqref{eq:6rs}.  Therefore,
\eqref{eq:boundnotunitinexregnewtstep} is violated for $j > \jir$,
which means that the line search must terminate with a step length
$\alpha_k = \theta^{j_k}$, with $1 \le j_k \le \jir+1$. The previous index 
$j=j_k-1$ satisfies \eqref{eq:boundnotunitinexregnewtstep}, so we have
\begin{equation*}
	\theta^{2j_k} \geq 
	\frac{3\theta^2}{L_H+\eta}(1-\zeta)\epsilon_H \|d_k\|^{-1},
\end{equation*}
so that 
\begin{align*}
	f(x_k) - f(x_k+\theta^{j_k}d_k) \; & \geq \; \frac{\eta}{6}\theta^{3 j_k}
	\|d_k\|^3 \\
	&  \geq \; \frac{\eta}{6}
	\left[\frac{3\theta^2}{L_H+\eta}(1-\zeta)\right]^{3/2}\epsilon_H^{3/2}
	\| d_k \|^{3/2} \\
	& \ge \frac{\eta}{6}
	\left[\frac{3\theta^2}{L_H+\eta}(1-\zeta)\right]^3\epsilon_H^3,
\end{align*}
where the final inequality follows from \eqref{eq:64A}, using the fact
that $\theta \in (0,1)$. Thus, condition\eqref{eq:inexactregnewtsteplength} 
also holds in the case of $\alpha_k<1$, and the proof is complete.
\end{proof}

\begin{theorem} \label{theo:wccitsinex}
Let Assumptions~\ref{assum:compactlevelset} and~\ref{assum:fC22}
hold. Then, Algorithm~\ref{algo:inexsorn} returns a point $x_k$
satisfying~\eqref{eq:eps2opt} in at most
\refer{
\begin{equation} \label{eq:wccitsinex}
	\hat{K} := \mathcal{\hat{C}}\max\left\{ \epsilon_g^{-3}\epsilon_H^{3},
	\epsilon_g^{-3/2},\epsilon_H^{-3}\right\}
\end{equation}
}
iterations, where
\begin{equation*}
	\mathcal{\hat{C}} := \frac{f(x_0)-\flow}{\hat{c}}, \quad
	\hat{c} := \min\left\{\frac{c_e}{8},c_g,\cin,\cir\right\},
\end{equation*}
with probability at least $1-\hat{K}\delta$. The constants $c_e$,
$c_g$, $c_{in}$, and $c_{ir}$ are defined in
Lemmas~\ref{lemma:eigsteplength}, \ref{lemma:gradsteplength}, 
\ref{lemma:inxactnewtsteplength}, and
\ref{lemma:inxactregnewtsteplength}, respectively.
\end{theorem}
\begin{proof}
For any iteration $l$ such that $x_l$ does not
satisfy~\eqref{eq:optwcc}, we must have that either $\min ( \| g_l \|,
\| g_{l+1} \|) > \epsg$ or $\lambdamin(\nabla^2 f(x_l)) < -\epsH$,
where the latter implies that $\lambda^i_l < -\tfrac12 \epsH$.  Thus,
similarly to the proof of Theorem~\ref{theo:itwcc}, we can consider
the following two cases.

\medskip

\textbf{Case 1: $\lambda^i_l < -\tfrac12 \epsH$.}

From Table~\ref{tab:stepsinexalgo}, we see that the same three choices
for $d_l$ as in the exact version are possible.  if $d_l =
\frac{R_l}{\|g_l\|}g_l$, we have exactly as in
Lemma~\ref{lemma:eigsteplength} that
\[
		f(x_l) - f(x_{l+1}) \; \ge \; c_e \epsH^3.
\]
When $d_l = -g_l/\|g_l\|^{1/2}$, we have from
Lemma~\ref{lemma:gradsteplength} that 
\begin{equation} \label{eq:roy.2}
		f(x_l) - f(x_{l+1}) \; \ge \; c_g \min\left\{
		\epsilon_g^{3}\epsilon_H^{-3},\epsilon_g^{3/2}\right\}.	
\end{equation}
The remaining case corresponds to the choice $d_l=v^i_l$. Since
\[
	\frac{d_l^\top \nabla^2 f(x_l) d_l}{\|d_l\|^2} = \lambda^i_l < 
	-\frac{\epsH}{2},
\]
with probability at least $1-\delta$ in that case, one can use the result of 
Lemma~\ref{lemma:eigsteplength} to deduce that
\[
	f(x_l) - f(x_{l+1})  \ge  c_e | \lambda^i_l|^3 \ge  \frac{c_e}{8}\epsH^3,
\]
again with probability at least $1-\delta$.

\medskip

\textbf{Case 2: $\lambda^i_l \ge -\tfrac12 \epsilon_H$, $\| g_l \| >
  \epsg$, and $\| g_{l+1} \| > \epsg$.} 
	
In this situation, we have three possible choices of search direction
$d_l$.  If $d_l = -g_l / \|g_l\|^{1/2}$, we have again from
Lemma~\ref{lemma:gradsteplength} that  \eqref{eq:roy.2} holds.
\refer{
If the inexact Newton direction is taken, we obtain by
Lemma~\ref{lemma:inxactnewtsteplength} that
\[
	f(x_k)-f(x_{k+1})  \; \ge \;
	\cin \min\left\{\epsg^3 \epsH^{-3},\epsH^3\right\}.
\]
}
Finally, if the search direction is the inexact regularized Newton
direction, that is, $d_l = d^{ir}_l$, we have from
Lemma~\ref{lemma:inxactregnewtsteplength} that
\refer{
\[
	f(x_k)-f(x_{k+1}) \; \ge \; \cir \min\left\{\epsg^3 \epsH^{-3},
	\epsH^3\right\}.
\]
}
By putting all these bounds together, as in the proof of
Theorem~\ref{theo:itwcc}, we obtain that the number of iterations
before reaching a point satisfying~\eqref{eq:optwcc} is bounded above
by $\hat{K}$ defined in the statement of the theorem.

Recalling that for each of these iterations, there is a probability
$\delta$ that the randomized Lanczos iteration in
\eqref{eq:inexeigenvec} will fail, we bound the probability of failure
during the course of the algorithm by $\hat{K} \delta$.
\end{proof}

Note that if $\delta$ is chosen large enough such that $1-\hat{K}
\delta<0$, Theorem~\ref{theo:wccitsinex} is not informative. The same
remark holds for the corollary below, that makes use of the results
from Sections~\ref{subsec:inexacteig} and~\ref{subsec:inexactnewt} to
obtain a bound on the total number of Hessian-vector multiplications
and gradient evaluations needed by the procedure (assuming that these
operations cost roughly the same).
\begin{corollary} \label{coro:computingtimeinex}
Suppose the assumptions of Theorem~\ref{theo:wccitsinex} hold, and let
$\delta \in (0,1)$ be given. Then the total number of gradient
evaluations and Hessian-vector multiplications required by
Algorithm~\ref{algo:inexsorn} to reach an iterate
satisfying~\eqref{eq:optwcc} is satisfied is bounded by
\begin{equation} \label{eq:computingtimeinex}
	\begin{array}{l}
	\left[2+\min\left\{n,\tfrac{1}{\sqrt{2}}(U_H+2)^{1/2}\epsilon_H^{-1/2}
	\ln\left(\frac{4(U_H+2)^{3/2}\epsilon_H^{-3/2}}
	{\zeta}\right)\right\} + \right.\\
	 \\
	\quad\quad\quad\quad\quad \left.\min\left\{n,
	(U_H+2)^{1/2}\epsilon_H^{-1/2}\frac{\ln(n/\delta^2)}{2}
	\right\}\right] \times \hat{K},
	\end{array}
\end{equation}
with probability $1-\hat{K} \delta$. 
\end{corollary}
\begin{proof}
The proof follows directly from Lemmas~\ref{lemma:Lanczosfixedproba}
and \ref{lemma:itsCGapprox}, setting $M=U_H+2$ and
$\varepsilon=\epsH/2$, noting that for both Newton and regularized
Newton steps, the condition number of the respective coefficient
matrices can be bounded by $(U_H+2)/\epsH$.
\end{proof}

As in Section~\ref{subsec:wccexactits}, we can particularize this
result to a specific choice of tolerances.
\begin{corollary} \label{coro:computingtimespecific}
Suppose that the assumptions of Theorem~\ref{theo:wccitsinex} hold,
and let $\delta \in (0,1)$ be given. Define $\epsg=\eps$ and $\epsH =
\sqrt{\eps}$, for some $\eps \in (0,1)$. Then the number of gradient
evaluations and Hessian-vector products needed to
Algorithm~\ref{algo:inexsorn} to satisfy \eqref{eq:optwcc} is bounded by 
\refer{ 
\begin{equation} \label{eq:computingtime7over4}
	\begin{array}{l}
	\left[2+\min\left\{n,\tfrac{1}{\sqrt{2}}(U_H+2)^{1/2}\epsilon^{-1/4}
	\ln\left(\frac{4(U_H+2)^{3/2}\epsilon^{-3/4}}
	{\zeta}\right)\right\} + \right.\\
	\\
	\quad\quad\quad\quad\quad \left.\min\left\{n,
	(U_H+2)^{1/2}\epsilon^{-1/4}\frac{\ln(n/\delta^2)}{2}
	\right\}\right] \times \hat{C} \epsilon^{-3/2},
	\end{array}
\end{equation}
where $\hat{C}$ is defined as in Theorem~\ref{theo:wccitsinex},
with probability at least $1-\hat{C} \eps^{-3/2} \delta$.
}
\end{corollary}

This result is meaningful when $\delta \ll \eps^{3/2}$. In terms of
the complexity bound, such a choice is not prohibitively small,
because $\delta$ enters into the bound \eqref{eq:computingtime7over4} only 
inside a $\log$ term.

We can obtain a bound for the case of $\delta=0$ (that is, 
almost certainty), at the cost of taking $n$ Lanczos
iterations whenever the smallest eigenvalue is needed (see
Lemma~\ref{lemma:Lanczosfixedproba}). In this case, the
bound~\eqref{eq:computingtime7over4} either becomes
$\mathcal{O}\left(\left(n+\ln\left(\epsilon^{-1}\right)\right)
\epsilon^{-7/4}\right)$ or $\mathcal{O}(n \epsilon^{-3/2})$, depending
on which term dominates in the quantity corresponding to conjugate
gradient iterations.

For very large $n$ and $\delta>0$, we can consider that the term
involving $\epsilon$ is smaller than $n$ in both minimum expressions
in Corollary~\ref{coro:computingtimespecific}. In this
case, the bound is
\begin{equation*}
	\mathcal{O}\left( 
	\ln\left(\frac{1}{\min\{\epsilon,\delta/\sqrt{n}\}}\right)
	\epsilon^{-7/4}\right).
\end{equation*}
This complexity matches other recent findings
~\cite{NAgarwal_ZAllenZhu_BBullins_EHazan_TMa_2017,
YCarmon_JCDuchi_OHinder_ASidford_2017}.

\refer{In terms of dependencies with respect to problem constants, we 
can reproduce the analysis from Section~\ref{subsec:wccexactits}, replacing 
$c_n$ and $c_r$ by $c_{in} = \mathcal{O}(L_H^{-3})$ and 
$c_{ir} = \mathcal{O}(L_H^{-3})$, respectively. For instance, the bound 
from the previous paragraph is in
\begin{equation*}
	\mathcal{O}\left( (f(x_0)-\flow) U_H^{1/2}L_H^3
	\ln\left(\frac{U_H}{\min\{\epsilon,\delta/\sqrt{n}\}}\right)
	\epsilon^{-7/4}\right).
\end{equation*}
We point out that the dependency on $L_H$ of our bound is worse than
those of~\cite{NAgarwal_ZAllenZhu_BBullins_EHazan_TMa_2017,
YCarmon_JCDuchi_OHinder_ASidford_2017}, due to the lack of explicit
use of this constant within our algorithm. Still, we believe our
dependency to match that of other Newton-type methods (although those are not 
enlightened in the related literature), and we consider such schemes as being 
more amenable to highly nonlinear settings where estimating such a constant 
would likely be impractical.}

As a final note, we observe that one could also include the number of 
line-search iterations into our complexity bound. However, this cost is 
essentially logarithmic in $1/{\epsilon}$, therefore it is dominated by 
the cost of the linear algebra techniques.

\section{Discussion}
\label{sec:discuss}

Among the many algorithmic frameworks that have been proposed for
smooth nonconvex optimization with second-order complexity guarantees,
it can be difficult to determine the algorithmic features that affect
the complexity analysis, or to understand how the guarantees provided 
by different algorithms relate to one another. We have presented a 
second-order complexity analysis of a framework that is based exclusively on 
line searches along certain directions. It does not require solution of
cubic-regularized or trust-region subproblems, or minimization of
convexified functions --- operations that are needed by other
approaches.  Our search directions are of several types --- gradient,
negative-curvature, Newton, and regularized Newton --- and we presented
a variant of our method that allows inexact direction computation
using iterative methods. We believe that ours is the first approach of
line-search type to achieve known optimal complexity, among methods
that identify points that satisfy approximate second-order necessary
conditions.

\refer{ In addition to the results of this paper, we observe that it
  is possible to modify our algorithms to attain points that satisfy
  termination conditions of the form~\eqref{eq:eps2opt} (rather
  than~\eqref{eq:optwcc}) by continuing to iterate in the situation in
  which $\|g_{k+1}\| \le \epsg$ but $\lambda_{min}(\nabla^2
  f(x_{k+1}))< -\epsH$. Step $k+1$ then yields a decrease that is a
  multiple of $\epsH^3$ (per Lemma~\ref{lemma:eigsteplength}), so the
  overall complexity estimates are preserved, even if step $k$ in this
  situation fails to produce a significant decrease in $f$.  }

In designing the framework of Algorithms~\ref{algo:sorn}
and~\ref{algo:inexsorn}, we have made some choices to give preference
to one direction choice over another, and we have also incorporated
several types of steps. Given the recent literature in this area, 
our proposed scheme is actually one particular instance of a broader
class of methods with similar complexity guarantees but possibly
diverse practical performance. An implementation of our approach would
raise several delicate issues, for example, issues associated with
failure of the randomized Lanczos procedure for obtaining an estimate
of the smallest eigenvalue. An incorrect estimate here could lead to
the conjugate gradient method subsequently being applied to an
indefinite matrix; a robust implementation would need to detect and
recover from such an occurrence.  Additionally, the choice of suitable
values for the bound on the Hessian norm is likely to be of critical
importance. Addressing these concerns in the aim of developing a
practical algorithm with good complexity guarantees is the subject of
ongoing research.

\section*{Acknowledgments} 
We are grateful to the anonymous referees and associate editor of the
original version of the paper, whose construtive comments led to
numerous improvements.

\appendix

\section{Technical Result} \label{app:A}

\refer{ We prove a technical result that is used in several
  proofs, including that of Lemma~\ref{lemma:regnewtsteplength}.
\begin{lemma} \label{lem:T1}
For positive scalars $a$ and $b$, and $t \ge 0$, we have
\[
-a + \sqrt{a^2+b t} \ge (-a + \sqrt{a^2+b}) \min(t,1).
\]
\end{lemma}
\begin{proof}
For the case of $t \ge 1$, we have
\[
	-a + \sqrt{a^2+b t} \ge -a + \sqrt{a^2+b}  = (-a + \sqrt{a^2+b}) \min(t,1),
\]
so the result holds in this case.  For $t \in (0,1)$, we need to show
that
\[
	-a + \sqrt{a^2+b t} \ge (-a + \sqrt{a^2+b}) t.
\]
This claim follows from the following chain of equivalences:
\begin{alignat*}{2}
&& -a + \sqrt{a^2+bt} &\ge (-a + \sqrt{a^2+b}) t \\
& \Leftrightarrow & \sqrt{a^2+bt} &\ge (-a + \sqrt{a^2+b}) t +a \\
& \Leftrightarrow &  a^2+bt & \ge  (-a + \sqrt{a^2+b})^2 t^2 + 2a  (-a + 
\sqrt{a^2+b})t + a^2 \\
& \Leftrightarrow & (b + 2a^2- 2a \sqrt{a^2+b})t & \ge
(-a + \sqrt{a^2+b})^2 t^2  \\
& \Leftrightarrow &  (-a + \sqrt{a^2+b})^2 t & \ge (-a + \sqrt{a^2+b})^2 t^2 \\
& \Leftrightarrow & 1 & \ge t,
\end{alignat*}
completing the proof.
\end{proof}
}

\bibliographystyle{siam}
\bibliography{refs3over2}

\end{document}